\documentclass[a4paper, 11pt]{amsart}
\usepackage[left=3cm,right=3cm]{geometry}

\usepackage[T1]{fontenc}
\usepackage[english]{babel}

\usepackage{amsthm, amssymb, amsfonts, amsmath, mathrsfs}
\usepackage[all]{xy}

\usepackage{caption}
\usepackage{url}
\usepackage{tikz}
\usepackage{hyperref}

\usepackage{xcolor}

\usepackage{array}
\newcommand{\PreserveBackslash}[1]{\let\temp=\\#1\let\\=\temp}
\newcolumntype{C}[1]{>{\PreserveBackslash\centering}p{#1}}
\newcolumntype{R}[1]{>{\PreserveBackslash\raggedleft}p{#1}}
\newcolumntype{L}[1]{>{\PreserveBackslash\raggedright}p{#1}}

\newcounter{stepcounter}
\theoremstyle{plain}

\newtheorem{thm}{Theorem}[section]

\newtheorem{lem}[thm]{Lemma}
\newtheorem{prop}[thm]{Proposition}
\newtheorem{cor}[thm]{Corollary}

\theoremstyle{definition}

\newtheorem{eg}[thm]{Example}
\newtheorem{defn}[thm]{Definition}

\newtheorem{remark}[thm]{Remark}

\date{}

\newcommand\bit{\begin{itemize}}
\newcommand\eit{\end{itemize}}
\newcommand\bet{\begin{enumerate}}
\newcommand\eet{\end{enumerate}}
\newcommand\ed{\end{document}}

\DeclareFontFamily{U}{mathx}{\hyphenchar\font45}
\DeclareFontShape{U}{mathx}{m}{n}{
      <5> <6> <7> <8> <9> <10>
      <10.95> <12> <14.4> <17.28> <20.74> <24.88>
      mathx10
      }{}
\DeclareSymbolFont{mathx}{U}{mathx}{m}{n}
\DeclareFontSubstitution{U}{mathx}{m}{n}
\DeclareMathAccent{\widecheck}{0}{mathx}{"71}
\DeclareMathAccent{\wideparen}{0}{mathx}{"75}

\newcommand{\e}{\varepsilon}
\newcommand\Om{\Omega}
\newcommand\del{\partial}
\newcommand\adel{\ol{\partial}}
\newcommand\DEL{\Delta}

\newcommand\bC{{\mathbb C}}

\newcommand\bR{{\mathbb R}}

\newcommand\EE{{\mathcal E}}
\newcommand\F{{\mathcal F}}
\newcommand\FF{{\mathcal F}}

\newcommand{\OO}{\mathcal{O}}

\def\CC{\mathbb{C}}

\newcommand\fg{{\mathfrak g}}

\newcommand\co{\mathrm{co}}

\newcommand\exd{\mathrm{d}}

\newcommand\haar{\mathrm{\bf h}}

\newcommand\unit{\mathrm{U}}

\newcommand\hw{\mathrm{hw}}
\newcommand\lw{\mathrm{lw}}
\newcommand\id{\mathrm{id}}
\newcommand\proj{\mathrm{proj}}

\newcommand\inv{^{-1}}

\newcommand\by{\times}

\newcommand\oby{\otimes}

\newcommand\wed{\wedge}
\newcommand\sseq{\subseteq}

\def\qbinom#1#2{\ensuremath{\left[\kern-.3em\left[\genfrac{}{}{0pt}{}{#1}{#2}\right]\kern-.3em\right]_q}}
\newcommand\ol{\overline}

\usepackage{tikz}
\usetikzlibrary{decorations.pathreplacing}

\usepackage[english]{babel}

\newcommand{\fl}{\mathfrak{l}}
\newcommand{\fc}{\mathfrak{c}}

\usepackage{mathtools}

\DeclareMathOperator{\trace}{tr}

\author[F. D\'iaz Garc\'ia]{Fredy D\'iaz Garc\'ia} 
\address{Mathematical Institute of Charles University, Sokolovsk\'a 83, Prague, Czech Republic}
\email{fdiaz@karlin.mff.cuni.cz}

\author[A. O. Krutov]{Andrey O. Krutov}
\address{Institute of Mathematics, Czech Academy of Sciences, \v{Z}itn\'a 25, 115 67 Prague, Czech Republic}  \email{krutov@math.cas.cz}

\author[R. \'O Buachalla]{R\'eamonn \'O Buachalla}
\address{Mathematical Institute of Charles University, Sokolovsk\'a 83, Prague, Czech Republic}
\email{obuachalla@karlin.mff.cuni.cz}

\author[P. Somberg]{Petr Somberg}
\address{Mathematical Institute of Charles University, Sokolovsk\'a 83, Prague, Czech Republic} \email{somberg@karlin.mff.cuni.cz}

\author[K. R. Strung]{Karen R. Strung}
\address{Institute of Mathematics, Czech Academy of Sciences, \v{Z}itn\'a 25, 115 67 Prague, Czech Republic} \email{strung@math.cas.cz}

\title[{Positive and Line Modules over $\OO_q(G/L_S)$ }]{\textbf{Positive Line Modules over  the irreducible quantum flag manifolds}}

\thanks{{\tiny FDG was partially funded by Conacyt (Consejo Nacional
de Ciencia y Tecnolog\'ia, M\'exico). AK was supported through the program ``Oberwolfach Leibniz Fellows'' by the Mathematisches Forschungsinstitut Oberwolfach in 2018 and by the QuantiXLie Centre of Excellence, a project cofinanced by the Croatian Government and European Union through the European Regional Development Fund---the Competitiveness and Cohesion Operational Programme (KK.01.1.1.01.0004). R\'OB acknowledges FNRS support through  a postdoctoral fellowship within the framework of the MIS Grant ``Antipode'' grant number F.4502.18. KRS was supported by a Radboud Excellence Initiative postdoctoral
fellowship, a Sonata 9 NCN grant 2015/17/D/ST1/02529. AK and KRS are currently funded by GA\v{C}R project 20-17488Y and RVO: 6798584. R\'OB and PS are partially supported by the grant GA\v{C}R $306-33/19
06357$. R\'OB and FDG are supported by the Charles University PRIMUS grant PRIMUS/21/SCI/026.}
}

  % Hopf algebras and their applications [See also 16S40, 57T05]

\begin{document}

\begin{abstract}
Noncommutative K\"ahler structures were recently introduced as a framework for studying noncommutative K\"ahler geometry on quantum homogeneous spaces. It was subsequently observed that the notion of a positive vector bundle directly generalises to this setting, as does the Kodaira vanishing theorem. In this paper, by restricting to covariant K\"ahler structures of irreducible type (those having an irreducible space of holomorphic $1$-forms) we provide simple cohomological criteria for positivity, allowing one to avoid explicit curvature calculations. These general results are  applied to our motivating family of examples, the irreducible  quantum flag manifolds $\OO_q(G/L_S)$. Building on the recently established noncommutative Borel--Weil theorem, every relative line module over $\OO_q(G/L_S)$ can be identified as positive, negative, or flat, and it is then concluded that each K\"ahler structure is of Fano type. 
\end{abstract}

\maketitle

%\tableofcontents

\section{Introduction}
Positive line bundles, which is to say, line bundles whose Chern curvature is a positive definite $(1,1)$-form, play a central role in modern complex geometry. Analogously, ample line bundles are fundamental objects of study in projective algebraic geometry.  An ample line bundle is a line bundle $\EE$ such that, for any coherent sheaf $\mathcal{S}$ and sufficiently large $k$, the tensor product $\mathcal{S} \otimes \EE^{\otimes k}$ is generated by global sections. Under the GAGA (\emph{g\'eometrie alg\'ebrique et g\'eom\'etrie analytique}) correspondence \cite{SerreGAGA}, positive and ample line bundles are equivalent.  The existence of positive line bundles has many remarkable implications for the structure of a complex manifold.  For example, the Kodaira embedding theorem says that a compact K\"ahler manifold is projective if and only if it admits a positive line bundle \cite[\textsection~5.3]{HUY}. Positivity also has significant cohomological implications, as evidenced by the celebrated Kodaira vanishing theorem and the subsequent slew of related vanishing theorems  \cite{VTBOOK}. Given  these elegant results, the natural impulse is to try to extend the concept of positivity to settings beyond ordinary complex geometry. This has been met with tremendous success in the study of varieties over fields of prime characteristic. Positivity, or rather in this case ampleness, has been key to understanding innate differences between these geometries, for example the failure of the Kodaira vanishing theorem in prime characteristic \cite{MR541027}. Another striking extension has been to the setting of noncommutative projective algebraic geometry, where ample sequences and ample pairs are by now considered foundational structures \cite{ArtinZhang}.

%%%%%%%%%%%%%%%%%%%%%%%%%%%%%%%%%%%%%%%%%%%%%%%
%%%%%%%%%%%%%%%%%%%%%%%%%%%%%%%%%%%%%%%%%%%%%%%
%%%%%%%%%%%%%%%%%%%%%%%%%%%%%%%%%%%%%%%%%%%%%%%
%%%%%%%%%%%%%%%%%%%%%%%%%%%%%%%%%%%%%%%%%%%%%%%
%%%%%%%%%%%%%%%%%%%%%%%%%%%%%%%%%%%%%%%%%%%%%%%
%%%%%%%%%%%%%%%%%%%%%%%%%%%%%%%%%%%%%%%%%%%%%%%
%%%%%%%%%%%%%%%%%%%%%%%%%%%%%%%%%%%%%%%%%%%%%%%
%% Positivity in Noncommutative Geometry 

The goal of this paper is to explore the idea of positivity for the noncommutative differential geometry of
quantum groups. In particular, we show that the relative line modules over an irreducible quantum flag
manifold endowed with its Heckenberger--Kolb calculus, are either positive, flat, or negative, 
\textbf{Theorems~\ref{thm:threeway}} and \textbf{\ref{thm:posnegLINEs}}.
Furthermore, we are able to distinguish between these three cases using cohomological information, \textbf{Corollary~\ref{cor:cohomology.pos.neg}}, in the form of the recently established noncommutative Akizuki--Nakano identities \cite[Corollary 7.8]{OSV}.

Positivity in noncommutative differential geometry is a concept that has been formulated only recently in the companion paper~\cite{OSV}. These two papers are part of a series exploring the noncommutative complex geometry of quantum homogeneous spaces \cite{CQHKS,DOKSS,MMF2,MMF3,OSV} based around the notion of a noncommutative K\"ahler structure, as introduced in \cite{MMF3}. In this context, the classical Koszul--Malgrange theorem~\cite{KoszulMalgrange}  allows for an obvious noncommutative generalisation of the definition of a holomorphic module. As in the classical setting, every Hermitian holomorphic module has a uniquely associated Chern connection \cite[Proposition 4.4]{BeggsMajidChern}.  In \cite{OSV} it is observed that the definition of a positive line bundle extends directly to the noncommutative setting, where we call them \emph{positive line modules}. Building on this observation,  a corresponding Kodaira vanishing theorem can be formulated  and the definition of a noncommutative K\"ahler structure can be refined to give the definition of a noncommutative Fano structure. As shown in \cite{OSV}, the implied vanishing of cohomologies for Fano structure makes it possible to calculate noncommutative holomorphic Euler characteristics.

%%%%%%%%%%%%%%%%%%%%%%%%%%%%%%%%%%%%%%%%%%%%%%%
%%%%%%%%%%%%%%%%%%%%%%%%%%%%%%%%%%%%%%%%%%%%%%%
%%%%%%%%%%%%%%%%%%%%%%%%%%%%%%%%%%%%%%%%%%%%%%%
%%%%%%%%%%%%%%%%%%%%%%%%%%%%%%%%%%%%%%%%%%%%%%%
%%%%%%%%%%%%%%%%%%%%%%%%%%%%%%%%%%%%%%%%%%%%%%%
%%%%%%%%%%%%%%%%%%%%%%%%%%%%%%%%%%%%%%%%%%%%%%%
%%%%%%%%%%%%%%%%%%%%%%%%%%%%%%%%%%%%%%%%%%%%%%%
%% Curvature from Cohomology

Despite an abundance of structure, calculating the curvature of a line module in the quantum  setting remains an extremely challenging task: classical tools are either not yet developed or are unavailable entirely. Any attempt at brute force calculations quickly becomes prohibitively lengthy and tedious. The complications involved can already be seen for the example of the standard Podle\'s sphere as discussed in \textbf{Example \ref{PodlesCurvature}}. Fortunately, the worst of these calculations can be avoided entirely by restricting to a particularly tractable subclass, which subsumes the quantum projective spaces: those covariant K\"ahler structures which are irreducible.

%%%%%%%%%%%%%%%%%%%%%%%%%%%%%%%%%%%%%%%%%%%%%%%
%%%%%%%%%%%%%%%%%%%%%%%%%%%%%%%%%%%%%%%%%%%%%%%
%%%%%%%%%%%%%%%%%%%%%%%%%%%%%%%%%%%%%%%%%%%%%%%
%%%%%%%%%%%%%%%%%%%%%%%%%%%%%%%%%%%%%%%%%%%%%%%
%%%%%%%%%%%%%%%%%%%%%%%%%%%%%%%%%%%%%%%%%%%%%%%
%%%%%%%%%%%%%%%%%%%%%%%%%%%%%%%%%%%%%%%%%%%%%%%
%%%%%%%%%%%%%%%%%%%%%%%%%%%%%%%%%%%%%%%%%%%%%%%
%%%% THE IRREDUCIBLE QUANTUM FLAG MANIFOLDS

Our motivating examples are the irreducible, or cominiscule, quantum flag manifolds $\OO_q(G/L_S)$. Forming a large and robust family, they are a natural class to consider when attempting to extend geometric notions from the classical to the noncommutative.  Indeed, it is becoming increasingly clear that the noncommutative geometry of the quantum flag manifolds is key to understanding the noncommutative geometry of quantum groups in general. Here, the necessary cohomological information is provided by the irreducible quantum flag manifold Borel--Weil theorem~\cite{IrrBW, KMOS}.  This allows us to prove in \textbf{Theorem~\ref{thm:HKFano}} that every irreducible quantum flag manifold, endowed with its Heckenberger--Kolb calculus, is of \emph{Fano type} in the sense of \cite[Definition 8.8]{OSV}.

%%%%%%%%%%%%%%%%%%%%%%%%%%%%%%%%%%%%%%%%
%%%%%%%%%%%%%%%%%%%%%%%%%%%%%%%%%%%%%%%%
%%%%%%%%%%%%%%%%%%%%%%%%%%%%%%%%%%%%%%%%
%%%%%%%%%%%%%%%%%%%%%%%%%%%%%%%%%%%%%%%%
%%%%%%%%%%%%%%%%%%%%%%%%%%%%%%%%%%%%%%%%
%%%%%%%%%%%%%%%%%%%%%%%%%%%%%%%%%%%%%%%%
%%%%%%%%%%%%%%%%%%%%%%%%%%%%%%%%%%%%%%%%
%%%%%%%%%%%%%%%%%%%%%%%%%%%%%%%%%%%%%%%%
%% Applications

The positivity results established in this paper have a number of important applications in other works. In \cite{OSV}  positivity, along with the  noncommutative Kodaira vanishing theorem, is used to prove a noncommutative generalisation of the Bott--Borel--Weil theorem for positive line modules.  In \cite{SpectralGap} positivity is used to prove a spectral gap for the negatively twisted Dolbeault--Dirac operators  over the irreducible quantum flag manifolds. This in turn allows for a proof that the closure of each such operator is Fredholm \cite{CQHKS}. This gives a particularly satisfying application of the machinery of classical complex geometry to the quantum world, showing how the spectrum of a noncommutative Dirac operator is shaped by the geometry of the underlying Heckenberger--Kolb $q$-deformed de Rham complex.

%%%%%%%%%%%%%%%%%%%%%%%%%%%%%%%%%%%%%%%%
%%%%%%%%%%%%%%%%%%%%%%%%%%%%%%%%%%%%%%%%
%%%%%%%%%%%%%%%%%%%%%%%%%%%%%%%%%%%%%%%%
%%%%%%%%%%%%%%%%%%%%%%%%%%%%%%%%%%%%%%%%
%%%%%%%%%%%%%%%%%%%%%%%%%%%%%%%%%%%%%%%%
%%%%%%%%%%%%%%%%%%%%%%%%%%%%%%%%%%%%%%%%
%%%%%%%%%%%%%%%%%%%%%%%%%%%%%%%%%%%%%%%%
%%%%%%%%%%%%%%%%%%%%%%%%%%%%%%%%%%%%%%%%

\subsection{Summary of the Paper}

The paper is organised as follows. In \textsection\ref{section:Prelims} we recall necessary background material, including noncommutative  K\"ahler structures, Hermitian and holomorphic modules, and compact quantum group algebras. In particular, we recall the recently introduced notion of a compact quantum homogeneous (CQH) K\"ahler space \cite[Definition 3.1]{CQHKS}, which details a natural set of compatibility conditions between covariant K\"ahler structures and compact quantum group algebras.

In  \textsection\ref{section:IrreducibleCQHHermitianSpaces} we develop the general theory of the paper. In particular, we introduce the notion of an irreducible CQH-K\"ahler space, and show that for any such space, a covariant Hermitian holomorphic line module is either positive, negative, or flat. We then build upon this result to show that we can distinguish between these three choices by examining the degree zero Dolbeault cohomology of the line module in question.

In \textsection\ref{section:DJ}  we present our  motivating family of examples, the irreducible quantum flag manifolds $\OO_q(G/L_S)$, their relative line modules $\EE_l$, for $l \in \mathbb{Z}$, along with their Heckenberger--Kolb calculi. We recall the irreducible CQH-K\"ahler structure of each $\OO_q(G/L_S)$, and the associated  noncommutative generalisation of the Borel--Weil theorem. We then build upon the general theory presented in \textsection3, and, for each $k \in \mathbb{Z}_{>0}$, prove that $\EE_k > 0$, and  $\EE_{-k} < 0$. As a consequence, we observe that the K\"ahler structure of the Heckenberger--Kolb calculus is a Fano structure.

For the reader's convenience we also include three short appendices. Appendix~\ref{app:A} presents Takeuchi's equivalence, the natural setting for discussing  homogeneous vector bundles in the noncommutative setting. Appendix \ref{app:Levi} contains a discussion of the definition of the quantum Levi group $\OO_q(L_S)$. Finally, Appendix~\ref{app:B} provides necessary technical details about the quantum flag manifolds and their canonical modules.

\subsubsection*{Acknowledgments} Part of this work was carried out when R\'OB and AK visited KRS at the Institute of Mathematics, Astrophysics and Particle Physics at Radboud University, and we thank the institute for their support. The authors thank the Universit\'e libre de Bruxelles for hosting FDG during September and October 2019.  KRS and R\'OB are also grateful for a visit to Mathematisches Forschungsintitut Oberwolfach in December 2018 where AK was a Leibniz fellow, and to the Mathematics Department at the University of Zagreb during July 2019.  All five authors benefited from meeting at the conference ``Quantum Flag Manifolds in Prague'' at the Charles University in September 2019. We also thank Matthias Fischmann, Vincent Grandjean, Dimitry Leites, Paolo Saracco, Jan \v{S}\v{t}ov\'i\v{c}ek,  Adam-Christaan van Roosmalen, and Elmar Wagner for helpful discussions

%%%%%%%%%%%%%%%%%%%%%%%%%%%%%%%%%%
%%%%%%%%%%%%%%%%%%%%%%%%%%%%%%%%%%
%%%%%%%%%%%%%%%%%%%%%%%%%%%%%%%%%%
%%%%%%%%%%%%%%%%%%%%%%%%%%%%%%%%%%
%%%%%%%%%%%%%%%%%%%%%%%%%%%%%%%%%%
%%%%%%%%%%%%%%%%%%%%%%%%%%%%%%%%%%
%%%%%%%%%%%%%%%%%%%%%%%%%%%%%%%%%%
%%%%%%%%%%%%%%%%%%%%%%%%%%%%%%%%%%

\section{Preliminaries} \label{section:Prelims}

We recall the basic definitions and results for differential calculi, complex structures, Hermitian structures, and K\"ahler structures, as well as noncommutative vector bundles over these objects.  For a more detailed introduction see \cite{MMF2,MMF3}, and references therein. For an excellent presentation of classical complex and K\"ahler geometry see \cite{HUY}. 

Throughout the paper all algebras are assumed to be unital and defined over  $\mathbb{C}$, and  all unadorned tensor products are defined over $\mathbb{C}$. We denote $\mathbf{i} := \sqrt{-1}$.

\subsection{Differential Calculi}

A {\em differential calculus}  $\big(\Om^\bullet \simeq \bigoplus_{k \in \mathbb{Z}_{\geq 0}} \Om^k, \exd\big)$ is a differential graded algebra which is generated as an algebra by the elements $a, \exd b$, for $a,b \in \Om^0$. We call an element $\omega \in \Omega^{\bullet}$ a \emph{form}, and if $\omega \in \Omega^k$, for some $k \in \mathbb{Z}_{> 0}$, then $\omega$ is said to be \emph{homogeneous} of degree $|\omega| := k$.  The product of two forms $\omega,\nu \in \Omega^{\bullet}$ is usually denoted by $\omega \wedge \nu$, unless one of the forms is of degree $0$, whereupon the product is denoted by juxtaposition. A \emph{differential calculus over} an algebra $B$ is a differential calculus such that $\Om^0 = B$. A differential calculus is said to have  \emph{total degree} $m \in \mathbb{Z}_{\geq 0}$, if $\Omega^m \neq 0$, and $\Om^k = 0$, for every $k > m$. 

A {\em differential $*$-calculus} over a $*$-algebra $B$ is a differential calculus over $B$ such that the \mbox{$*$-map} of $B$ extends to a (necessarily unique) conjugate-linear involutive map $*:\Om^\bullet \to \Om^\bullet$ satisfying $\exd(\omega^*) = (\exd \omega)^*$, and 
\begin{align*}
\big(\omega \wed \nu\big)^*  =  (-1)^{kl} \nu^* \wed \omega^*, &  & \text{ for all } \omega \in \Om^k, \, \nu \in \Om^l. 
\end{align*}
We say that $\omega \in \Omega^\bullet$ is \emph{closed} if $\exd \omega = 0$, and \emph{real} if $\omega^*= \omega$. For further details in differential calculi we refer to \cite{BeggsMajid:Leabh}.

\subsection{Complex Structures}

We now recall the definition of a complex structure as introduced in \cite{BS,KLvSPodles}. This generalises the properties of the de Rham complex of a classical complex manifold \cite[\textsection 2.6]{HUY}, and was motivating  by the pioneering examples of the Podle\'s calculus of standard Podle\'s sphere as described in \cite{Maj}, and more generally the Heckenberger--Kolb calculi \cite{HKdR}.

\begin{defn}
A {\em complex structure} $\Om^{(\bullet,\bullet)}$ for a differential $*$-calculus $(\Om^{\bullet},\exd)$ is a choice of \mbox{$\mathbb{Z}_{\geq 0}^2$-algebra} grading $\bigoplus_{(a,b)\in \mathbb{Z}_{\geq 0}^2} \Om^{(a,b)}$ for $\Om^{\bullet}$ such that
\begin{align*}
1. ~ \Om^k = \bigoplus_{a+b = k} \Om^{(a,b)}, & & 2. ~
\big(\Om^{(a,b)}\big)^* = \Om^{(b,a)}, & & 3. ~ \exd \Om^{(a,b)} \subseteq \Omega^{(a+1,b)} \oplus \Omega^{(a,b+1)},
\end{align*}
for all $k \in \mathbb{Z}_{\geq 0}$, and $(a,b) \in \mathbb{Z}^{2}_{\geq 0}$. We call an element of $\Om^{(a,b)}$ an $(a,b)$-form.
\end{defn}

Note that for any complex structure there is a degree $(1,0)$-map $\del$, and a degree \mbox{$(0,1)$-map} $\adel$, uniquely defined by
$
\exd = \del + \adel.
$
The triple 
$
(\Om^{(\bullet,\bullet)}, \del,\ol{\del})
$
forms a double complex, which we call the \emph{Dolbeault double complex}. Moreover, both $\del$ and $\adel$ satisfy the graded Leibniz rule, and $\del(\omega^*) = (\adel \omega)^*$, for all $\omega \in \Omega^{\bullet}$. 

We say that a complex structure $\Om^{(\bullet,\bullet)}$ is {\em factorisable} if $B$-bimodule isomorphisms
\begin{align*}
1. ~ \Om^{(a,0)} \oby_B \Om^{(0,b)} \simeq \Om^{(a,b)}, & & 2. ~ \Om^{(0,b)} \oby_B \Om^{(a,0)} \simeq \Om^{(a,b)}, & & \textrm{ for all } (a,b) \in \mathbb{Z}^2_{\geq 0},
\end{align*}
are given by the multiplication map of $\Omega^{\bullet}$. Recall that every complex manifold is automatically factorisable \cite[\textsection1.2]{HUY}, as are the Heckenberger--Kolb calculi of the irreducible quantum flag manifolds (see  \textsection\ref{subsect:hk}).

\subsection{Hermitian and K\"ahler Structures} \label{subsection:HandKS}

In this subsection we recall the general theory of Hermitian and K\"ahler structures as introduced in \cite{MMF3}. 

\begin{defn} Let $\Om^{\bullet}$  be a differential $*$-calculus over a $*$-algebra $B$, of even total degree $2n$. An {\em Hermitian structure} for $\Om^{\bullet}$  is a pair $(\Om^{(\bullet,\bullet)}, \sigma)$  consisting of a complex structure  $\Om^{(\bullet,\bullet)}$ and  a central real $(1,1)$-form $\sigma$, called the {\em Hermitian form}, such that, for the {\em Lefschetz operator} $L:\Om^\bullet \to \Om^\bullet$, defined by $L(\omega) := \sigma \wedge \omega$, isomorphisms are given by
\begin{align*}
L^{n-k}: \Om^{k} \to  \Om^{2n-k}, & & \text{ for all } k = 0, \dots, n-1.
\end{align*}
A \emph{K\"ahler structure} is an Hermitian structure such that the Hermitian form is closed, in which case we call it a {\em K\"ahler form}.
\end{defn} 
The \emph{$(a,b)$-primitive forms} of an Hermitian structure are
\begin{align*}
P^{(a,b)} : = \begin{cases} 
      \{\alpha \in \Om^{(a,b)} \,|\, L^{n-a-b+1}(\alpha) = 0\}, &  \text{ ~ if } a+b \leq n,\\
      0, & \text{ ~ if } a+b > n,
   \end{cases}
\end{align*} 
and we write $P^k := \bigoplus_{a+b = k} P^{(a,b)}$. %, and $P^{\bullet} := \bigoplus_{k \in \mathbb{Z}_{\geq 0}}  P^k$. 
The existence of an Hermitian structure implies  a direct generalisation of the classical Lefschetz decomposition \cite[Proposition 4.3]{MMF3}:
\begin{align*}
\Om^{k} \simeq \bigoplus_{j \geq 0} L^j\big(P^{k-2j}\big).
\end{align*}
In classical Hermitian geometry, the Hodge map of an Hermitian metric is related to the associated Lefschetz decomposition through the well-known Weil formula \cite[Proposition 1.2.31]{HUY}. In the noncommutative setting we take the Weil formula for our definition of the Hodge map: The {\em Hodge map} associated to an Hermitian structure $\big(\Om^{(\bullet,\bullet)},\sigma \big)$ is the $B$-bimodule map $\ast_{\sigma}: \Omega^\bullet \to \Omega^\bullet$ satisfying, for any $j \in \mathbb{Z}_{\geq 0}$,
\begin{align*}
\ast_{\sigma}\big(L^j(\omega)\big) = (-1)^{\frac{k(k+1)}{2}}\mathbf{i}^{a-b}\frac{j!}{(n-j-k)!}L^{n-j-k}(\omega), & & \textrm{ for } \omega \in P^{(a,b)}, \, a+b = k.
\end{align*}
The Hodge map allows us to construct a sesquilinear map (conjugate in the second variable) called the \emph{Hermitian metric}, 
\begin{align*}
g_{\sigma}: \Omega^{\bullet} \times \Omega^{\bullet} \to B, & & 
(\omega, \nu) \mapsto \left\{ \begin{array}{cc} *_{\sigma} (\omega \wedge *_\sigma(\nu^*)), &\text{ if } k=l,\\
0, & \text{ if } k\neq l. \end{array} \right. 
\end{align*}
It follows from \cite[Proposition 5.6]{OSV} that
\[ 
g_{\sigma}(\omega, \nu) =  g_{\sigma}(\nu, \omega)^*, \qquad \text{ for all } \omega, \nu \in \Omega^{\bullet}.
\] 
With respect to the Hermitian metric, the Lefschetz map $L$ is adjointable, and we denote its adjoint by $L^{\dagger} = \Lambda$. The map $\Lambda$ can be explicitly presented as
\[ 
\Lambda = *_\sigma^{-1} \circ L \circ *_{\sigma}.
\]
We define the \emph{counting operator} $H : \Omega^\bullet \to \Omega^\bullet$ by $H(\omega) := (k-n) \omega$, for $\omega \in \Omega^k$. Together the maps $L$, $\Lambda$, and $H$ give a representation of $\mathfrak{sl}_2$. (See \textsection \ref{subsection:AkiZNak} for the general twisted version of this representation.)

%%%%%%%%%%%%%%%%%%%%%%%%%%%%%%%%%%%%%%%%%%%%%%%%%%%%%%%%%%%%
%%%%%%%%%%%%%%%%%%%%%%%%%%%%%%%%%%%%%%%%%%%%%%%%%%%%%%%%%%%%
%%%%%%%%%%%%%%%%%%%%%%%%%%%%%%%%%%%%%%%%%%%%%%%%%%%%%%%%%%%%
%%%%%%%%%%%%%%%%%%%%%%%%%%%%%%%%%%%%%%%%%%%%%%%%%%%%%%%%%%%%
%%%%%%%%%%%%%%%%%%%%%%%%%%%%%%%%%%%%%%%%%%%%%%%%%%%%%%%%%%%%
%%%%%%%%%%%%%%%%%%%%%%%%%%%%%%%%%%%%%%%%%%%%%%%%%%%%%%%%%%%%
%%%%%%%%%%%%%%%%%%%%%%%%%%%%%%%%%%%%%%%%%%%%%%%%%%%%%%%%%%%%
%%%%%%%%%%%%%%%%%%%%%%%%%%%%%%%%%%%%%%%%%%%%%%%%%%%%%%%%%%%%
%%%%%%%%%%%%%%%%%%%%%%%%%%%%%%%%%%%%%%%%%%%%%%%%%%%%%%%%%%%%
%%%%%%%%%%%%%%%%%%%%%%%%%%%%%%%%%%%%%%%%%%%%%%%%%%%%%%%%%%%%
%%%%%%%%%%%%%%%%%%%%%%%%%%%%%%%%%%%%%%%%%%%%%%%%%%%%%%%%%%%%
%%%%%%%%%%%%%%%%%%%%%%%%%%%%%%%%%%%%%%%%%%%%%%%%%%%%%%%%%%%%
%%%%%%%%%%%%%%%%%%%%%%%%%%%%%%%%%%%%%%%%%%%%%%%%%%%%%%%%%%%%
%%%%%%%%%%%%%%%%%%%%%%%%%%%%%%%%%%%%%%%%%%%%%%%%%%%%%%%%%%%%

\subsection{Holomorphic Modules}

In this subsection we present the notion of a noncommutative holomorphic module, as  has been considered in various places, for example \cite{BS}, \cite{PolishSch}, and \cite{KLvSPodles}.  Motivated by the Serre--Swan theorem, we consider projective modules as noncommutative analogues of vector bundles. In particular, a \emph{line module} over $B$ will be an invertible $B$-bimodule $\EE$, where \emph{invertible} means that there exists another $B$-bimodule ${}^{\vee}\EE$ such that 
$
\EE \otimes_B {}^{\vee}\EE \simeq {}^{\vee}\EE \otimes_B \EE \simeq B.
$
Note that any such $\EE$ is automatically projective as a left (and as a right) $B$-module.

We can define a noncommutative analogue of a holomorphic vector bundle via the classical Koszul--Malgrange characterisation of holomorphic bundles \cite{KoszulMalgrange}. (See \cite{OSV} for a more detailed discussion.) For $\Omega^\bullet$ a differential calculus over an algebra $B$, and $\mathcal{F}$ a left $B$-module, a \emph{connection} on $\FF$ is a $\mathbb{C}$-linear map $\nabla:\mathcal{F} \to \Omega^1 \otimes_B \FF$ satisfying 
\begin{align*}
\nabla(bf) = \exd b \otimes f + b \nabla f, & & \textrm{ for all } b \in B, f \in \FF.
\end{align*}
With respect to a choice $\Omega^{(\bullet,\bullet)}$ of complex structure on $\Omega^{\bullet}$, a \emph{$(0,1)$-connection for $\mathcal{F}$} is a connection with respect to the differential calculus $(\Omega^{(0,\bullet)},\adel)$.

Any connection can be extended to a map $\nabla: \Omega^\bullet \otimes_B \mathcal{F} \to   \Omega^\bullet \otimes_B \mathcal{F}$ uniquely defined by
\begin{align*}
\nabla(\omega \otimes f) =   \exd \omega \otimes f + (-1)^{|\omega|} \, \omega \wedge \nabla f, & & \textrm{for } f \in \FF, \, \omega \in \Omega^{\bullet},
\end{align*}
for a homogeneous form $\omega$ with degree  $|\omega|$. The \emph{curvature} of a connection is the left $B$-module map $\nabla^2: \mathcal{F} \to \Omega^2 \otimes_B \mathcal{F}$. A connection is said to be {\em flat} if $\nabla^2 = 0$. Since $\nabla^2(\omega \otimes f) = \omega \wedge \nabla^2(f)$, a connection is flat if and only if  the pair $(\Omega^\bullet \otimes_B \FF, \nabla)$ is a complex. 

\begin{defn}
Fix a differential $*$-calculus $\Omega^{\bullet}$, over a $*$-algebra $B$, endowed with a choice of complex structure $\Omega^{(\bullet,\bullet)}$. A \emph{holomorphic module over $B$} is then a pair $(\mathcal{F},\adel_{\mathcal{F}})$, where  $\mathcal{F}$ is a finitely generated projective left $B$-module, and the map $\adel_{\mathcal{F}}: \mathcal{F} \to \Omega^{(0,1)} \otimes_B \mathcal{F}$ is a flat $(0,1)$-connection, which we call the \emph{holomorphic structure} for $(\FF, \adel_{\FF})$. 
\end{defn}

Note that for any fixed $a \in \mathbb{Z}_{\geq 0}$, a holomorphic module $(\FF,\adel_{\FF})$ has an associated complex
$$
\adel_{\FF}: \Omega^{(a,\bullet)} \otimes_B \FF  \to \Omega^{(a,\bullet)} \otimes_B \FF.
$$
For any $b \in \mathbb{Z}_{\geq 0}$, we denote by $H^{(a,b)}_{\adel}(\FF)$ the $b^{\mathrm{th}}$-cohomology group of this complex.

%%%%%%%%%%%%%%%%%%%%%%%%%%%%%%%%%%%%%%%%%%%%%%%%%%%%%%%%%%%%
%%%%%%%%%%%%%%%%%%%%%%%%%%%%%%%%%%%%%%%%%%%%%%%%%%%%%%%%%%%%
%%%%%%%%%%%%%%%%%%%%%%%%%%%%%%%%%%%%%%%%%%%%%%%%%%%%%%%%%%%%
%%%%%%%%%%%%%%%%%%%%%%%%%%%%%%%%%%%%%%%%%%%%%%%%%%%%%%%%%%%%
%%%%%%%%%%%%%%%%%%%%%%%%%%%%%%%%%%%%%%%%%%%%%%%%%%%%%%%%%%%%
%%%%%%%%%%%%%%%%%%%%%%%%%%%%%%%%%%%%%%%%%%%%%%%%%%%%%%%%%%%%
%%%%%%%%%%%%%%%%%%%%%%%%%%%%%%%%%%%%%%%%%%%%%%%%%%%%%%%%%%%%
%%%%%%%%%%%%%%%%%%%%%%%%%%%%%%%%%%%%%%%%%%%%%%%%%%%%%%%%%%%%
%%%%%%%%%%%%%%%%%%%%%%%%%%%%%%%%%%%%%%%%%%%%%%%%%%%%%%%%%%%%
%%%%%%%%%%%%%%%%%%%%%%%%%%%%%%%%%%%%%%%%%%%%%%%%%%%%%%%%%%%%
%%%%%%%%%%%%%%%%%%%%%%%%%%%%%%%%%%%%%%%%%%%%%%%%%%%%%%%%%%%%
%%%%%%%%%%%%%%%%%%%%%%%%%%%%%%%%%%%%%%%%%%%%%%%%%%%%%%%%%%%%
%%%%%%%%%%%%%%%%%%%%%%%%%%%%%%%%%%%%%%%%%%%%%%%%%%%%%%%%%%%%
%%%%%%%%%%%%%%%%%%%%%%%%%%%%%%%%%%%%%%%%%%%%%%%%%%%%%%%%%%%%

\subsection{Holomorphic Hermitian Modules} \label{subsection:HVBS}

In this subsection, we assume that $B$ is a \mbox{$*$-algebra}, and generalise to our noncommutative setting the classical notion of a holomorphic Hermitian vector bundle. %(See Remark \ref{remark:ClassicalHVB} below for a discussion of the relationship with the classical case.)

For any left $B$-module, denote by $^\vee\!\FF := \mathrm{Hom}_B(\FF, B)$ the dual module, which is a right $B$-module with respect to pointwise multiplication 
\begin{align*}
\phi b (f) := \phi(f) b, & & \phi \in \! \, ^\vee\!\FF, & & \text{for } b \in B,  f \in \FF. 
\end{align*}  
Moreover, we denote by $\overline{\FF}$ the \emph{conjugate right $B$-module} of $\FF$, as defined by the action
\begin{align*}
\overline{\FF} \otimes B \to \overline{\FF}, & & \overline{f} \otimes b \mapsto \overline{b^*f}. 
\end{align*}
For a $*$-algebra $B$, the \emph{cone of positive elements} $B_{\geq0}$ is defined by
\[ 
B_{\geq  0} := \left\{\sum_{1 \leq i \leq l}  b_i^* b_i  \mid b_i \in B, \, l \in \mathbb{Z}_{\geq 0} \right\}\!.
\]
We denote the non-zero positive elements of $B$ by $B_{>0} := B_{\geq 0} \setminus \{0\}$. 

\begin{defn} \label{defn:Hermitianmetric} 
An \emph{Hermitian module} over a $*$-algebra $B$ is a pair $(\F,h_{\F})$, where $\F$ is a finitely generated projective left $B$-module and $h_{\F}:\overline{\F} \to {}^{\vee}\!\F$ is a right $B$-module isomorphism, such that, for  the associated sesquilinear pairing, 
\begin{align*}
h_{\F}(-,-): \F \times \F \to B, & & (f,k) \mapsto h_{\F}(\overline{k})(f)
\end{align*}
it holds  that, for all non-zero $f,k \in \F$,
\begin{align*}
1. ~ h_{\F}(f,k) = h_{\F}(k,f)^*, & &  2. ~ h_{\F}(f,f) \in B_{>0}.
\end{align*}
\end{defn}

\begin{eg} \label{eg:HermModCalc}
Let $(\Omega^{(\bullet,\bullet)}, \sigma)$ be an Hermitian structure for a differential calculus $\Omega^{\bullet}(B)$. If $\Omega^\bullet$ is finitely generated and projective as a left $B$-module, and  we assume that  
\begin{align*}
g_{\sigma}(\omega, \omega) \in B_{>  0}, &&  \textrm{ for all non-zero } \omega \in \Omega^{\bullet},
\end{align*}
then the pair $(\Omega^{\bullet},g_{\sigma})$ is an Hermitian module. In what follows, when we say that the pair $(\Omega^{\bullet},g_{\sigma})$ associated to an Hermitian structure is an Hermitian module, we mean it in this sense.
\end{eg}

Let $(\FF,h_{\FF})$ be an Hermitian module, and consider the sesquilinear map
\begin{align*}
\mathfrak{h}_{\FF}: \Om^{\bullet} \otimes_B \F \times \Om^{\bullet} \otimes_B \F \to \Om^{\bullet}, & & (\omega \otimes f, \nu \otimes g)  \mapsto  \omega h_\F(f,g) \wedge \nu^*.
\end{align*}
A connection $\nabla\colon \F \to \Om^1 \otimes_B\F$ is \emph{Hermitian} if
\begin{align*}
\exd {\mathfrak{h}}_\mathcal{F}(f, g) = \mathfrak{h}_\F(\nabla(f), 1 \otimes {g}) + \mathfrak{h}_\F(1 \otimes  f,  {\nabla}({g})), && \textrm{ for all } f,g \in \F.
\end{align*}
%\end{defn}

\begin{defn}
A \emph{holomorphic Hermitian module} is a triple $(\mathcal{F},h_\F,\adel_{\mathcal{F}})$ such that $(\F, h_\F)$ is an Hermitian module and $(\F, \adel_{\mathcal{F}})$ is a holomorphic module.
\end{defn}
As established in \cite{BeggsMajidChern} (see also \cite{OSV}), for any  Hermitian holomorphic module $(\mathcal{F},h_\F,\adel_{\mathcal{F}})$, there exists a unique Hermitian connection $\nabla:\mathcal{F} \to  \Omega^1 \otimes_A \mathcal{F}$ satisfying
\begin{align*}
 \adel_{\F} = \big(\mathrm{proj}_{\Omega^{(0,1)}} \otimes_B \id\big) \circ  \nabla, 
\end{align*}
where $\proj_{\Omega^{(0,1)}}:\Omega^1 \to \Omega^{(0,1)}$ is the obvious projection. We call $\nabla$ the \emph{Chern connection} of $(\mathcal{F},h_\F, \adel_{\mathcal{F}})$, and denote
\begin{align*}
\del_{\F} := \big(\mathrm{proj}_{\Omega^{(1,0)}} \otimes_B \id\big) \circ  \nabla,
\end{align*}
where $\proj_{\Omega^{(1,0)}}:\Omega^1 \to \Omega^{(1,0)}$ is the obvious projection.

We finish this subsection with the notion of positivity for a holomorphic Hermitian module, motivated by the classical notion of positivity \cite[Proposition 5.3.1]{HUY}. It was first introduced in \cite[Definition 8.2]{OSV} and details a compatibility between Hermitian holomorphic modules and K\"ahler structures.

\begin{defn}\label{defn:positiveNegativeVB}
Let $\Omega^{\bullet}$ be a differential calculus over a $*$-algebra $B$, and let $(\Omega^{(\bullet,\bullet)},\kappa)$ be a K\"ahler structure for $\Omega^{\bullet}$. 
An Hermitian  holomorphic module $(\mathcal{F},h_{\F}, \adel_{\mathcal{F}})$  is said to be \emph{positive}, written $\F > 0$, if there exists a $\theta \in \mathbb{R}_{>0}$, such that the Chern connection $\nabla$ satisfies
\begin{align*}
\nabla^2(f) = -\theta \mathbf{i} \kappa \otimes f, & & \textrm{ for all } f \in \mathcal{F}.
\end{align*} 
Analogously, $(\mathcal{F}, h_{\F}, \adel_{\mathcal{F}})$ is said to be {\em negative}, written $\F <0$, if there exists a $\theta \in \mathbb{R}_{>0}$, such that the Chern connection $\nabla$ of ${\mathcal{F}}$ satisfies
\begin{align*}
\nabla^2(f) =  \theta \mathbf{i} \kappa \otimes f, & & \textrm{ for all } f \in \mathcal{F}.
\end{align*} 
\end{defn}

An important point to note is that for any factorisable complex structure $\Omega^{(\bullet,\bullet)}$ of total degree $2n$, the pair $\left(\Omega^{(n,0)},\adel\right)$ is a holomorphic module. Moreover, for a \emph{factorisable Hermitian structure}, or \emph{factorisable K\"ahler structure}, which is to say an Hermitian, or K\"ahler, structure whose constituent complex structure is factorisable, the triple $\left(\Omega^{(n,0)}, g_{\sigma},\wedge^{-1} \circ \adel\right)$ is an  Hermitian holomorphic module.

\subsection{Some Identities} \label{subsection:AkiZNak}

Suppose that $(\Omega^{(\bullet, \bullet)}, \sigma)$ is an Hermitian structure for a differential $*$-calculus over a $*$-algebra $B$, and that $\F$ is a left $B$-module. Consider the triple of operators on $\Omega^{\bullet} \otimes_{B} \F$:
\begin{align*}
L_\F := L \otimes \id_\F, & & H_F := H \otimes \id_\F, & &  \Lambda_\F := \Lambda \otimes \id_\F. 
\end{align*}
As established in \cite{OSV}, it holds that
\begin{align*}
& [H_\F, L_\F ] = 2L_\F,  \qquad [L_\F, \Lambda_\F ] = H_\F, \qquad  [H_\F, \Lambda_\F] = -2 \Lambda_\F,
\end{align*}
meaning that we have a representation of $\mathfrak{sl}_2$.

%%%%%%%%%%%%%%%%%%%%%%%%%%%%%%%%%%%%%%%%%%%%%%%%%%
%%%%%%%%%%%%%%%%%%%%%%%%%%%%%%%%%%%%%%
%%%%%%%%%%%%%%%%%%%%%%%%%%%%%%%%%%%%%%%%%%%%%%%%%%
%%%%%%%%%%%%%%%%%%%%%%%%%%%%%%%%%%%%%%%%%%%%%%%%%%
%%%%%%%%%%%%%%%%%%%%%%%%%%%%%%%%%%%%%%%%%%%%%%%%%%
%%%%%%%%%%%%%%%%%%%%%%%%%%%%%%%%%%%%%%%%%%%%%%%%%%
%%%%%%%%%%%%%%%%%%%%%%%%%%%%%%%%%%%%%%%%%%%%%%%%%%
%%%%%%%%%%%%%%%%%%%%%%%%%%%%%%%%%%%%%%%%%%%%%%%%%%
%%%%%%%%%%%%%%%%%%%%%%%%%%%%%%%%%%%%%%%%%%%%%%%%%%
%%%%%%%%%%%%%%%%%%%%%%%%%%%%%%%%%%%%%%%%%%%%%%%%%%
%%%%%%%%%%%%%%%%%%%%%%%%%%%%%%%%%%%%%%
%%%%%%%%%%%%%%%%%%%%%%%%%%%%%%%%%%%%%%%%%%%%%%%%%%
%%%%%%%%%%%%%%%%%%%%%%%%%%%%%%%%%%%%%%%%%%%%%%%%%%
%%%%%%%%%%%%%%%%%%%%%%%%%%%%%%%%%%%%%%%%%%%%%%%%%%
%%%%%%%%%%%%%%%%%%%%%%%%%%%%%%%%%%%%%%%%%%%%%%%%%%
%%%%%%%%%%%%%%%%%%%%%%%%%%%%%%%%%%%%%%%%%%%%%%%%%%
%%%%%%%%%%%%%%%%%%%%%%%%%%%%%%%%%%%%%%%%%%%%%%%%%%
%%%%%%%%%%%%%%%%%%%%%%%%%%%%%%%%%%%%%%%%%%%%%%%%%%
%%%%%%%%%%%%%%%%%%%%%%%%%%%%%%%%%%%%%%%%%%%%%%%%%%
%%%%%%%%%%%%%%%%%%%%%%%%%%%%%%%%%%%%%%
%%%%%%%%%%%%%%%%%%%%%%%%%%%%%%%%%%%%%%%%%%%%%%%%%%
%%%%%%%%%%%%%%%%%%%%%%%%%%%%%%%%%%%%%%%%%%%%%%%%%%
%%%%%%%%%%%%%%%%%%%%%%%%%%%%%%%%%%%%%%%%%%%%%%%%%%
%%%%%%%%%%%%%%%%%%%%%%%%%%%%%%%%%%%%%%%%%%%%%%%%%%
%%%%%%%%%%%%%%%%%%%%%%%%%%%%%%%%%%%%%%%%%%%%%%%%%%
%%%%%%%%%%%%%%%%%%%%%%%%%%%%%%%%%%%%%%%%%%%%%%%%%%
%%%%%%%%%%%%%%%%%%%%%%%%%%%%%%%%%%%%%%%%%%%%%%%%%%

\section{Irreducible CQH-Hermitian Spaces and Positive Line Modules} \label{section:IrreducibleCQHHermitianSpaces}

Determining positivity, or negativity, of an Hermitian holomorphic line module ostensibly requires one to calculate the Chern curvature explicitly. In practice, this can prove to be a very challenging technical task. This is true in the classical setting, and even more so in the noncommutative world, as can be seen for the simplest case of the standard Podle\'s sphere as presented in Example~\ref{PodlesCurvature}. However, as we demonstrate in this section, by imposing an irreducibility condition on our CQH-Hermitian structure, vanishing, and non-vanishing, of  zeroth cohomology groups can be used to conclude positivity. 

\subsection{Covariant Hermitian Structures} \label{subsect:covHERMstructure}

 Suppose that $A$ is a Hopf algebra and $B$ is a left $A$-comodule algebra.  A differential calculus $\Omega^\bullet$ over $B$ is said to be \emph{covariant} if the coaction $\Delta_L : B \to A \otimes B$ extends to a (necessarily unique) map $\Delta_L : \Omega^\bullet \to A \otimes \Omega^\bullet$ giving $\Omega^\bullet$ the structure of an $A$-comodule algebra, and such that $\exd$ is a left $A$-comodule map.
 
A complex structure for $\Omega^\bullet$ is said to be \emph{covariant} if the \mbox{$\mathbb{Z}^2_{\geq 0}$-decomposition} of $\Omega^{\bullet}$ is a decomposition in the category of left \mbox{$A$-comodules} $^A \text{Mod}$, which is to say, $\Omega^{(a,b)}$ is a left $A$-sub-comodule of $\Omega^\bullet$, for each $(a, b) \in \mathbb{Z}^2_{\geq 0}$. This implies that $\del$ and $\adel$ are left $A$-comodule maps.  

If $(\Omega^{(\bullet, \bullet)}, \sigma)$  is an Hermitian structure for $\Omega^\bullet$ such that $\Omega^{(\bullet, \bullet)}$ is a covariant complex structure and $\sigma$ is left $A$-covariant, that is, $\Delta_L(\sigma) = 1 \otimes \sigma$, then we say that   $(\Omega^{(\bullet, \bullet)}, \sigma)$ is a \emph{covariant Hermitian structure}. In this case, $L$, $*_\sigma$ and $\Lambda$ are also left $A$-comodule maps.  A \emph{covariant K\"ahler structure} is a covariant Hermitian structure which is moreover a K\"ahler structure.

We also require the notion of covariance for holomorphic modules, and for simplicity restrict to the setting of a quantum homogeneous space $B = A^{\co(H)}$ (see Appendix \ref{app:A}). A \emph{holomorphic relative Hopf module} is a holomorphic module $(\F, \adel_\F)$ over $B$ such that $\F$ is an object in $^A_B \text{mod}_0$ (see Appendix~\ref{app:TAK}) and $\adel_\F : \F \to \Omega^{(0,1)} \otimes_B \F$ is a left $A$-comodule map.

An \emph{Hermitian relative Hopf module} is an Hermitian module $(\F, h_\F)$ such that $\F$ is an object in $^A_B \text{mod}_0$ and such that the map
\begin{align*}
\overline{\F} \to \, ^{\vee}\!\F, & & \overline{f} \mapsto h_\F( \cdot, f)
\end{align*}
is a morphism in $^A_B\mathrm{mod}_0$, where the conjugate $\overline{\F}$ and dual ${}^{\vee}\!\F$ modules are understood as objects in $^A_B\mathrm{mod}_0$ in the sense of Appendix \ref{app:A}. Note that for any simple object $\F$, its conjugate $\overline{\F}$ and its dual ${}^{\vee}\!\F$ will again be simple, implying that  any covariant Hermitian structures will be unique up to positive scalar multiple.

If $(\F, h_\F, \del_F)$ is an Hermitian holomorphic relative Hopf module, then the Chern connection is always a left $A$-comodule map, see \cite[\textsection7.1]{OSV}. Finally, an \emph{Hermitian}, or \emph{holomorphic, relative line module}, is an Hermitian, or holomorphic, relative Hopf module $\EE$ such that $\dim(\Phi(\EE)) = 1$, where $\Phi$ is the functor from Takeuchi's equivalence given in Appendix~\ref{app:TAK}. 

\begin{remark} \label{remark:FibrewisePD}
In the relative Hopf module case, the  differential calculus is automatically finitely generated and projective. Hence, as discussed in Example \ref{eg:HermModCalc} we automatically get an Hermitian module.

In the commutative case, the induced sesquilinear pairing on the fibre of an  Hermitian relative Hopf module $(\F,h_{\F})$ is positive definite if and only if 
\begin{align} \label{eqn:FibrewisePD}
\e(h_{\F}(f,f)) \neq 0, & & \textrm{ for all non-zero } f \in \F.
\end{align}
Indeed, in the noncommutative setting, Takeuchi's equivalence allows us to make sense of fibrewise positive definiteness for an Hermitian relative Hopf module $(\F,h_{\F})$. From the discussion of \cite[\textsection 5.2]{MMF3}, it is clear that $h_{\F}$ is fibrewise positive definite if and only if \eqref{eqn:FibrewisePD} holds.
\end{remark}

%%%%%%%%%%%%%%%%%%%%%%%%%%%%%%%%%%%%%%%%%%%%%%%%%%%%%%%%%%%%

\subsection{CQH-Hermitian and CQH-K\"ahler Spaces}

We recall the definition of a compact quantum group algebra  \cite{KoornDijk}, the algebraic counterpart of Woronowicz's $C^*$-algebraic notion of a compact quantum group \cite{WoroCQPGs}.  A \emph{cosemisimple} Hopf algebra is a Hopf algebra endowed with a (necessarily unique) linear map $\mathbf{h} : A \to \mathbb{C}$, which we call the \emph{Haar functional}, satisfying $\mathbf{h}(1) = 1$, and
$$ 
(\id \otimes \mathbf{h}) \circ \Delta(a) = \mathbf{h}(a) 1, \qquad (\mathbf{h} \otimes \id) \circ \Delta(a) = \mathbf{h}(a) 1.
$$
A {\em compact quantum group algebra}, or a {\em CQGA}, is a cosemisimple Hopf $*$-algebra $A$  such that  $\haar(a^*a) > 0$, for all non-zero $a \in A$. A {\em CQGA-homogeneous space}  is a quantum homogeneous space such that  $A$ and $H$ are both CQGAs and $\pi : A \to H$ is a $*$-algebra map.

Let $(\Omega^{(\bullet,\bullet)}, \sigma)$ be an Hermitian structure of total degree $2n$, for $n \in \mathbb{Z}_{>0}$. The linear map
\begin{align*}
\int := \haar \circ \ast_{\sigma}: \Omega^{2n} \to \mathbb{C}
\end{align*}
is called the \emph{integral}. If $\int  \mathrm{d}  \omega = 0$, for all $\omega \in \Omega^{2n-1}$, then $(\Omega^{(\bullet,\bullet)}, \sigma)$ is said to be $\int$-closed.  Note that this is a special case of an orientable differential calculus with closed integral \cite[\textsection3.2]{MMF3}, where the Hodge map is taken as the orientation, and it generalises the classical situation of a manifold without boundary.

\begin{defn}
A \emph{compact quantum  homogeneous Hermitian space}, or simply a  \emph{CQH-Hermitian space}, is a quadruple $\mathbf{H} := \big(B = A^{\co(H)}, \Omega^\bullet, \Omega^{(\bullet,\bullet)},\sigma \big)$  where
\bet

\item $B = A^{\mathrm{co} (H)}$  is a CQGA-homogeneous space,

\item $\Omega^\bullet$ is a left $A$-covariant differential $*$-calculus over $B$, and an object in  $^A_B \mathrm{mod}_0$, 

\item $\big(\Omega^{(\bullet,\bullet)},\sigma\big)$ is a covariant  $\int$-closed Hermitian structure such that the pair $(\Omega^{\bullet},g_{\sigma})$ is an Hermitian module.

\eet
We denote by $\dim(\mathbf{H})$ the total degree of the constituent differential calculus $\Om^\bullet$, and call it the {\em dimension} of $\mathbf{H}$.  A CQH-K\"ahler space $\mathbf{K} := \big(B, \Omega^\bullet, \Omega^{(\bullet,\bullet)},\kappa \big)$ is a CQH-Hermitian space such that $(\Omega^{(\bullet, \bullet)}, \kappa)$ is a K\"ahler structure. 
\end{defn}

Over any CQH-Hermitian space $\mathbf{H} := \big(B = A^{\co(H)} , \Omega^\bullet, \Omega^{(\bullet,\bullet)},\sigma \big)$, every $\F \in \, ^A _B \text{mod}_0$ admits an Hermitian structure.  For a given  Hermitian module $(\mathcal{F}, h_{\F}, \adel_{\F})$ over $B$, an inner product is given by 
\begin{equation*} 
\langle \cdot, \cdot \rangle_{\F} : \F \by \F \to \mathbb{C}, \qquad (f,g) \mapsto \mathbf{h}\! \left(h_{\F}(f, g)\right)\!.
\end{equation*}
In particular, for the Hermitian module $\Omega^{\bullet} \otimes_B \F$, we have the inner product
\begin{equation} \label{OmegaF:inner product} 
\langle \cdot, \cdot \rangle_{\sigma, \F} : \Omega^{\bullet} \otimes_B \F \by \Omega^{\bullet} \otimes_B \F \to \mathbb{C}, \qquad (\omega \otimes f, \nu \otimes g) \mapsto \mathbf{h} \circ g_{\sigma}(\omega h_{\F}(f,g),\nu).
\end{equation}

We  denote by  $\del^{\dagger}_{\F}$, and $\adel^{\dagger}_{\F}$ the adjoint operators of $\del_{\F}$  and $\adel_{\F}$, respectively. That $\del$ and $\adel$ are adjointable with respect to this inner product is a consequence of the fact that they are covariant \cite[Proposition 5.15]{OSV}. We refer to any such operator as a {\em codifferential}.  Just as in the classical case \cite[\textsection4.1]{HUY}, each noncommutative codifferential admits a description in terms of the Hodge map \cite[Proposition 5.15]{OSV}. Since such formulae will not be needed in what follows, we recall only the case where $\F = B$, as originally established in  \cite[\textsection5.3.3]{MMF3}:
\begin{align}   \label{eqn:dagger}
\del^{\dagger} = - *_{\sigma} \circ \, \adel \circ *_{\sigma}, & & \adel^{\dagger} = - *_{\sigma} \circ\, \del \circ *_{\sigma}. 
\end{align}

The holomorphic, and anti-holomorphic, \emph{Laplace} operators on $(\F, \adel_\F)$ are defined  respectively  by 
 \begin{align*} 
 \Delta_{\adel_{\F}} :=  \adel^{\dagger}_{\F} \adel_\F + \adel_\F \adel^{\dagger}_\F, & &  \Delta_{\del_{\F}} :=  \del^{\dagger}_{\F} \del_\F + \del_\F \del^{\dagger}_\F.
  \end{align*}
As established in \cite[Corollary 7.8]{OSV}, the classical relationship between the Laplacians $\DEL_{\del_{\F}}$ and $\DEL_{\adel_{\F}}$ carries over to the noncommutative setting. Explicitly, the \emph{Akizuki--Nakano identity}
\begin{align} \label{eqn:ANid}
\DEL_{\adel_{\F}} = \DEL_{\del_{\F}} + [\mathbf{i} \nabla^2, \Lambda_\F].
\end{align}
carries over to the noncommutative setting.
The \emph{harmonic elements} of $(\F, \adel_\F)$ are defined  by 
 \begin{align*} 
 \mathcal{H}^{\bullet}_{\adel}(\F) := \ker(\Delta_{\adel_\F}).
  \end{align*}
The following noncommutative generalisation of classical Hodge decomposition was established in \cite[Theorem 6.4]{OSV}: Let $(\F, h, \adel_\F)$ be an Hermitian holomorphic module over a CQH-Hermitian space $(B, \Omega^{\bullet}, \Omega^{(\bullet, \bullet)}, \sigma)$. Then an orthogonal decomposition of \mbox{$A$-comodules} with respect to the Hermitian metric is given by 
$$
\Omega^{(0, \bullet)} \otimes_B \F= \mathcal{H}^{(0,\bullet)}_{\adel}(\F) \oplus \adel_\F( \Omega^{(0, \bullet)} \otimes_B \F) \oplus \adel_\F^\dagger( \Omega^{(0, \bullet)} \otimes_B \F).
$$
An isomorphism with cohomology is given by the projection
$$
\mathcal{H}^{(0,\bullet)}_{\adel}(\F) \to H^{(0,\bullet)}_{\adel}(\F), \qquad \alpha \mapsto [\alpha].
$$
Note that, in the untwisted case (which is to say, the case where we do not tensor $\Omega^\bullet$ with an Hermitian module $\F$) the Laplacian operators coincide, and hence by the Hodge identification of harmonic forms and cohomology classes, the holomorphic and anti-holomorphic cohomology groups coincide \cite[Corollary 7.7]{MMF3}.

\subsection{Irreducible CQH-Hermitian Spaces} 

Let $\mathbf{H} = \left(B=A^{\co(H)}, \Omega^{\bullet},\Omega^{(\bullet,\bullet)}, \sigma\right)$ be a  CQH-Hermitian space. Since  $\Omega^{(1,0)}$ and $\Omega^{(0,1)}$ are objects in $\,^A_B\mathrm{mod}_0$, we can ask if they are irreducible as objects in that category. We claim that $\Omega^{(1,0)}$ is irreducible if and only if $\Omega^{(0,1)}$ is irreducible. Indeed, for any proper non-trivial sub-object $N \subset \Omega^{(1,0)}$, it is clear that $N^* := \{ \omega^* \,|\, \omega \in N \}$ is a proper non-trivial sub-object of~$\Omega^{(0,1)}$. Clearly, we can use the same argument in the opposite direction, which proves the claim. This leads to the next definition.

\begin{defn} \label{defn:irrHS}
A CQH-Hermitian space $\mathbf{H} = \left(B=A^{\co(H)}, \Omega^{\bullet},\Omega^{(\bullet,\bullet)}, \sigma\right)$ is said to be {\em irreducible}  if $\Omega^{(1,0)}$, or equivalently $\Omega^{(0,1)}$, is irreducible as an object in $\,^A_B\mathrm{mod}_0$.  
\end{defn}

Irreducible Hermitian structures  generalise our motivating family of examples, the irreducible quantum flag manifolds $\OO_q(G/L_S)$, as presented in \textsection\ref{section:DJ}. Many of the properties of the irreducible quantum flag manifolds extend to this more general setting. %In the following theorem, we present those relevant to the sequel. 

\begin{lem}\label{lem:thethm}  Let $\mathbf{H} = \left(B = A^{\co(H)}, \Omega^{\bullet},\Omega^{(\bullet,\bullet)},\sigma\right)$ be a factorisable irreducible CQH-Hermitian space, and let $\EE$ be a relative line module over $B$.
\begin{enumerate} 

\item If $\Omega^{(0,1)}$ and $B$ are non-isomorphic as objects in   $^A_B\textrm{mod}_0$, then there exists at most one left $A$-covariant $(0,1)$-connection for $\EE$.

\item If $\Phi(\Omega^{(0,1)})$ is not self-dual as a left $H$-comodule, and $\adel_{\EE}$ is a left $A$-covariant $(0,1)$-connection for $\EE$, then $\adel_{\EE}$ is automatically a holomorphic structure for $\EE$.

\item The space of left $A$-coinvariant $(1,1)$-forms is a one-dimensional space spanned by~$\sigma$, that is, 
$$
\,^{\mathrm{co}(A)}\!\left(\Omega^{(1,1)}\right) = \CC \sigma.
$$

\item For any  Hermitian holomorphic relative line module $(\mathcal{E},h_{\EE},\adel_{\EE})$ over $\mathbf{H}$, there exists a scalar  $\vartheta \in \mathbb{R}$ such that 
\begin{align*}
\nabla^2(e) =\vartheta \mathbf{i } \sigma \otimes e, &&  \text{ \,for all }   e \in \mathcal{E},
\end{align*}
where $\nabla$ is the Chern connection of $(\mathcal{E},h_{\EE},\adel_{\EE})$.

\item  For $\EE, \, \nabla$,  and $\vartheta$ as in \emph{(iv)}, and denoting $2n := \dim(\mathbf{H})$, it holds that
\begin{align*}
[\Lambda_{\EE}, \nabla^2](e) = \Lambda_{\EE} \circ \nabla^2(e) =  \vartheta n e, && \text{for all }   e \in \mathcal{E}.
\end{align*}

\item In the K\"ahler setting, the Akizuki--Nakano identity \eqref{eqn:ANid} simplifies to 
\begin{align*}
    \Delta_{\adel_{\EE}} = \Delta_{\del_{\EE}} + \vartheta n \,\id
\end{align*}
for any relative line module $\EE$.
\end{enumerate}
\end{lem}
\begin{proof}
~~ {} \newline
(i) Assume that there exists a non-zero morphism $f:\EE \to \Omega^{(0,1)} \otimes_B \EE$, or equivalently, denoting $V := \Phi(\EE)$,  a non-zero morphism $\phi:V \to \Phi(\Omega^{(0,1)}) \otimes V$. This would imply the existence of a non-zero morphism 
\begin{align*}
 \mathbb{C} \simeq V \otimes V^* \to \Phi(\Omega^{(0,1)}) \otimes V \otimes V^* \simeq \Phi(\Omega^{(0,1)}),
\end{align*}
where $V^*$ denotes the dual left $H$-comodule of $V$. By irreducibility of $\Phi(\Omega^{(0,1)})$ this is necessarily an isomorphism. Since this contradicts our assumption that 
$
\Omega^{(0,1)} \not \simeq B,
$
we must conclude that there exists no such morphism $f$.

Let us now assume that there exists another  covariant $(0,1)$-connection $\delta_{\EE}:\EE \to \Omega^{(0,1)} \otimes_B \EE$ distinct from $\adel_{\EE}$. Then $\adel_{\EE} - \delta_{\EE}$ is a non-trivial left $B$-module map, and hence a non-trivial morphism in $^A_B\textrm{mod}_0$. Since we have shown that no such morphism exists, we must conclude that no such $\delta_{\EE}$ exists, which is to say $\adel_{\EE}$ is the unique covariant $(0,1)$-connection for $\EE$.   

% \item
\vspace{0.25cm}
\hspace{0.5cm}  (ii)
Assuming that  $\Phi(\Omega^{(0,1)})$ is not self-dual implies that there can exist no copy of 
$\mathbb{C}$ in $\Phi(\Omega^{(0,1)}) \otimes \Phi(\Omega^{(0,1)})$. Splitting the multiplication map gives an embedding  of $\Phi(\Omega^{(0,2)})$ into  $\Phi(\Omega^{(0,1)}) \otimes \Phi(\Omega^{(0,1)})$, which means that there is no copy of 
$\mathbb{C}$ in $\Phi(\Omega^{(0,2)})$. 

Let us now assume that $\adel_{\EE}$ is not flat. Since $\adel_{\EE}^2$ is a left $B$-module map from $\EE$  to $\Omega^{(0,2)} \otimes_B \EE$, it is automatically a non-zero morphism in $^A_B\textrm{mod}_0$. This means that we have a non-zero morphism $V \to \Phi(\Omega^{(0,2)}) \otimes V$, and hence a morphism 
$$
\mathbb{C} \simeq V \otimes V^* \to \Phi(\Omega^{(0,2)}) \otimes V \otimes V^* \simeq \Phi(\Omega^{(0,2)}).
$$
Since this contradicts our assumption that $\Phi(\Omega^{(0,1)})$ is not self-dual, we are forced to conclude that $\adel_{\EE}$ is necessarily flat.

% \item \noindent
\vspace{0.25cm}
\hspace{0.5cm}  (iii)
Since $^H\text{mod}$ is a rigid monoidal category, $V$ is invertible.  
In particular 
$
 V \otimes V^* \simeq \mathbb{C},
$
By assumption $\Omega^{(\bullet, \bullet)}$ is factorisable, hence $\Phi(\Omega^{(1,1)})$ is isomorphic to  $\Phi(\Omega^{(1,0)}) \otimes \Phi(\Omega^{(0,1)})$. Denote the decomposition of $\Phi(\Omega^{(1,1)})$ into irreducible comodules by 
\begin{align*} 
& ~~~\Phi(\Omega^{(1,1)}) \simeq  \Phi(\Omega^{(1,0)}) \otimes \Phi(\Omega^{(0,1)}) \simeq: \bigoplus_{i} K_i. 
\end{align*}
Since $\sigma$ is a left $A$-coinvariant Hermitian form, we have $[\sigma] \in \, ^{\mathrm{co}(A)}(\Phi(\Omega^{(1,1)}))$, implying that one of the summands $K_i$ must be isomorphic to the trivial comodule. Moreover, since both $\Phi(\Omega^{(1,0)})$ and $\Phi(\Omega^{(0,1)})$ are by assumption irreducible, $\Phi(\Omega^{(1,0)})$ and $\Phi(\Omega^{(0,1)})$ are dual and precisely one of the summands will be trivial. For $\unit : \F \mapsto \Psi \circ \Phi(\F)$ the unit of Takeuchi's equivalence, it is easily seen that 
\begin{align*}
\unit\!\left(\!^{\,\co(A)}(\Omega^{(1,1)})\right) =  1 \otimes \left(^{\co(H)}(\Phi(\Omega^{(1,1)}))\right)  \simeq   1 \otimes \mathbb{C},
\end{align*}
giving the claimed equality.

% \item[(iv),(v)]
\vspace{0.25cm}
\hspace{0.5cm}  (iv), (v) Consider the space of left $H$-comodule maps $V \to \Phi(\Omega^{(1,1)}) \otimes V$, which is clearly isomorphic to the space of left $H$-comodule maps
$$
\mathbb{C} \simeq V \otimes V^* \to \Phi(\Omega^{(1,1)}) \otimes V \otimes V^* \simeq \Phi(\Omega^{(1,1)}). 
$$
Since ${}^{\mathrm{co}(A) }(\Omega^{(1,1)}) = \CC \sigma$, these spaces must be one-dimensional. Now the curvature operator is a morphism in $^A_B\text{mod}_0$, and so, we can consider its image under Takeuchi's functor $\Phi$. As a morphism from $V$ to $\Phi(\Omega^{(1,1)}) \otimes V$, it must be of the form
\begin{align*}
\Phi(\nabla^2)([e]) = \alpha [\sigma] \otimes [e], & & \textrm{ for some } \alpha \in \mathbb{C}.
\end{align*}

Consider now the commutative diagram:
\begin{align*}
\xymatrix{ 
\Omega^{(1,1)} \otimes_B \EE  & & & &  \ar[llll]_{~\,\unit^{-1} \,~~ } A \square_H  \Phi\!\left(\Omega^{(1,1)} \otimes_B \EE\right) \\       
\EE  \ar[u]^{\nabla^2}   \ar[rrrr]_{\unit ~~~ }  & & & & A \square_H \Phi(\EE) \ar[u]_{\id \otimes \Phi(\nabla^2)}.
}
\end{align*}
For any particular element $e \in \EE$, it holds that 
\begin{align*}
\nabla^2(e) = \unit^{-1} \circ (\id \otimes \Phi(\nabla^2)) \circ \unit(e) = & \, \unit^{-1} \circ (\id \otimes \Phi(\nabla^2))(e_{(-1)} \otimes [e_{(0)}]) \\
= & \, \alpha  \unit^{-1}(e_{(-1)} \otimes [\sigma \otimes e_{(0)}])\\
= & \, \alpha e_{(-2)}S(e_{(-1)})\sigma \otimes e_{(0)}\\
= & \, \alpha \e(e_{(-1)}) \sigma \otimes e_{(0)}\\
= & \, \alpha \sigma \otimes e,
\end{align*}
where we have used the formula for the inverse of Takeuchi's unit, as presented in \eqref{eqn:unitinverse}.

The operators $\Delta_{\del_{\EE}}$ and $\Delta_{\adel_{\EE}}$ are, by construction, self-adjoint operators on $\Omega^{\bullet} \otimes_B \EE$. Thus any eigenvalue of $\Delta_{\del_{\EE}} - \Delta_{\adel_{\EE}}$ must be a real scalar. It now follows from the Akizuki--Nakano identity that
\begin{align*}
(\Delta_{\del_{\EE}} - \Delta_{\adel_{\EE}})(e) =   [\mathbf{i} \nabla^2, \Lambda_{\EE}](e)
                                                                           =   - \mathbf{i} \Lambda_{\EE} \circ \nabla^2(e)
                                                                           =   - \alpha \mathbf{i} \Lambda_{\EE}(\sigma \otimes e)\
                                                                           =   -  \alpha \mathbf{i} \, \Lambda_{\EE} \circ L_{\EE}(e).
\end{align*}

Recalling now the twisted Lefschetz identities, and denoting $2n := \dim(\mathbf{H})$, the above expression can be reduced to  
\begin{align*}
-  \alpha \mathbf{i}  \Lambda_{\EE} \circ L_{\EE}(e)
                                                                           =  ~ \alpha \mathbf{i} \, [L_{\EE}, \Lambda_{\EE}](e)
                                                                           =  ~ \alpha \mathbf{i} H_{\EE}(e)
                                                                           =  -  \mathbf{i} n \alpha  e.
\end{align*}
Thus $\alpha \mathbf{i} \in \mathbb{R}$ and setting $\vartheta := -\alpha \mathbf{i}$ gives the equation in \textrm{(iv)} as claimed, establishing (v) in the process.

\vspace{0.25cm}
\hspace{0.5cm}  (vi) The simplified form of the Akizuki--Nakano identity follows directly from (v).
\end{proof}

For the special case of a CQH-K\"ahler space, Lemma~\ref{lem:thethm} (iv) implies the following result. %This serves as the principal theoretical result of the paper.

\begin{thm} \label{thm:threeway}
For any Hermitian holomorphic relative line module $\EE$ over a factorisable irreducible  CQH-K\"ahler space, precisely one of the following three possibilities holds:
\begin{enumerate}
\item $\mathcal{E} > 0$,
\item $\mathcal{E}$  \emph{ is flat}, 
\item $\mathcal{E}<0$.
\end{enumerate}
\end{thm}

The general cohomological consequences of positivity presented in the Kodaira vanishing theorem now allow us to produce sufficient cohomological conditions for positivity, flatness, or negativity, of a line module over an irreducible CQH-K\"ahler space.

\begin{cor} \label{cor:cohomology.pos.neg} Let $\EE$  be an Hermitian holomorphic relative line module over an irreducible CQH-K\"ahler space $\mathbf{K}$. 
\begin{enumerate}
\item If $H^0_{\del}(\EE)$ and $H^0_{\adel}(\EE)$ are both non-zero, then $\EE$ is flat, and $H^0_{\del}(\EE) = H^0_{\adel}(\EE)$.
\item  If $H^0_{\adel}(\EE) \neq 0$ and $H^0_{\adel}(\EE) \neq H^0_{\del}(\EE)$, then $\EE$ is positive and $H^0_{\del}(\EE) =0$.
\item  If $H^0_{\del}(\EE) \neq 0$ and $H^0_{\adel}(\EE) \neq H^0_{\del}(\EE)$, then $\EE$ is negative and $H^0_{\adel}(\EE) =0$.
\end{enumerate}
\end{cor}
\begin{proof}
\hspace{0.5cm}  (i) By part (vi) of Lemma \ref{lem:thethm}, for any eigenvector $e$ of $\Delta_{\del_{\EE}}$, with eigenvalue $\lambda$, 
$$
\Delta_{\adel_{\EE}}(e) = \Delta_{\del_{\EE}}(e) + \vartheta n e = (\lambda + \vartheta n) e.
$$
Since $\Delta_{\del_{\EE}}$ is a positive operator $\lambda \in \mathbb{R}_{\geq 0}$. Thus, if $\EE < 0$, which is to say, if $\vartheta > 0$,  we would have that $\lambda + \vartheta n > 0$, implying that $\Delta_{\adel_{\EE}}(e) \neq 0$. In particular, we see that if $\mathcal{E} < 0$, then by the isomorphism of harmonic elements and cohomology classes $H^0_{\adel}(\mathcal{E}) = 0$. An analogous argument implies that if $\mathcal{E} > 0$, then $H^0_{\del}(\mathcal{E}) = 0$.

\vspace{0.25cm}

\hspace{0.5cm}  From these considerations we see that if $H^0_{\del}(\EE)$ and $H^0_{\adel}(\EE)$ are both non-zero, then $\EE$ can be neither positive nor negative. It now follows from Theorem~\ref{thm:threeway} that $\EE$ is flat. Finally, we see that if $\mathcal{E}$ is flat, then since the Laplacians $\Delta_{\adel_{\EE}}$ and  $\Delta_{\del_{\EE}}$ coincide, the associated cohomology groups must also coincide.

\vspace{0.25cm}

\hspace{0.5cm}  (ii) From the argument of (i) we see that if $H^0_{\adel}(\EE) \neq 0$, then $\EE$ cannot be negative. Moreover, since the holomorphic and anti-holomorphic cohomology groups are not equal, $\EE$ cannot be flat. Theorem~\ref{thm:threeway} now implies that $\EE$ is positive. Finally, it follows from (i) that if  $H^0_{\del}(\EE)$ were non-zero, then $\EE$ would be flat. Thus to avoid contradiction we must assume that $H^0_{\del}(\EE) = 0$.

\vspace{0.25cm}

\hspace{0.5cm}  (iii) The argument for this case is exactly analogous to the argument of (ii).
\end{proof}

\begin{eg} 
We now consider an interesting family of examples, namely  irreducible CQH-K\"ahler spaces~$\mathbf{K} = (B = A^{\co(H)},
\Omega^{\bullet},\Omega^{(\bullet,\bullet)},\kappa)$  of dimension $4$. In particular, we bear in mind the quantum projective plane $\mathcal{O}_q(\mathbb{CP}^2)$, a member of the general family of examples discussed in \textsection\ref{section:DJ}. Note that in dimension $4$ the Hodge map satisfies~$\ast_\kappa^2 = \id$ on $2$-forms. Hence~$\ast_\kappa$
has eigenvalues~$1$ and~$-1$. It follows from the definition 
of~$\ast_\kappa$ that~$\ast_\kappa(\kappa) = \kappa$. Following the classical definition we define the \emph{first Chern class}~$c_1(\EE)$, of a Hermitian holomorphic relative line module~$\EE$, to be $c_1(\EE) := \trace \nabla^2$, where $\trace$ is the trace of $\nabla^2$ easily defined using Takeuchi's equivalence.  By
Lemma~\ref{lem:thethm} the  first Chern class $c_1(\EE)$ is proportional to~$\kappa$, and so,  it satisfies 
\[
  \ast_{\kappa}(c_1(\EE)) = c_1(\EE).
\]
In the classical setting  such connections are called \emph{self-dual connections}. They are of interest because they
satisfy the Yang--Mills equations~\cite{AHDM}. For a more detailed discussion for the special case of
$\OO_q(\mathbb{CP}^2)$, see the pair of papers~\cite{DAndreaLandiMonopoles,Antiselfdual}. 
\end{eg}

\begin{remark}
Any positive, negative, or flat  relative Hopf module $(\F,h,\adel_{\F})$ over an irreducible CQH-K\"ahler space~$\mathbf{K} = (B = A^{\co(H)},\Omega^{\bullet},\Omega^{(\bullet,\bullet)},\kappa)$ satisfies
\begin{align} \label{eqn:HEinstein}
 \Lambda_{\F} \circ \nabla^2 =  \vartheta \mathbf{i} \, \mathrm{id}_{\F}, & & \text{ for some } \vartheta  \in \mathbb{R}.
\end{align}
This is established by a verbatim extension of the arguments of Lemma \ref{lem:thethm} to the relative Hopf module setting.  What is interesting about \eqref{eqn:HEinstein} is that it is satisfied by a far larger class of Hermitian holomorphic modules than those which are  positive, negative, or flat. Classically, this motivates the definition of an Hermite--Einstein module \cite[\textsection4.B]{HUY}. We observe that this definition carries over directly to the noncommutative setting: For a CQH-K\"ahler space $\mathbf{K}$, we say that an Hermitian holomorphic module $(\mathcal{F},h,\adel_{\F})$ over $\mathbf{K}$ is {\em Hermite--Einstein} if 
\begin{align*}
\Lambda_{\F} \circ \nabla^2 = \gamma \mathbf{i} \, \mathrm{id}_{\F}, & & \text{ for some } \gamma  \in \mathbb{R}.
\end{align*}
%%%%%%%%%%%%%%%%%%%%%%%%

Hermite--Einstein vector  bundles are objects of central importance in classical complex geometry. They are intimately related to the theory of Yang--Mills connections \cite{HEYMs}. Moreover, the Donaldson--Uhlenbeck--Yau theorem relates the existence of an Hermite--Einstein metric to semi-stability of the vector bundle, see \cite{KobyHit}. The investigation of how such results extend to the  noncommutative setting presents itself as a very interesting direction for future research.
\end{remark}

%%%%%%%%%%%%%%%%%%%%%%%%%%%%%%%%%%%%%%%%%%%%%%%%%%
%%%%%%%%%%%%%%%%%%%%%%%%%%%%%%%%%%%%%%%%%%%%%%%%%%
%%%%%%%%%%%%%%%%%%%%%%%%%%%%%%%%%%%%%%%%%%%%%%%%%%
%%%%%%%%%%%%%%%%%%%%%%%%%%%%%%%%%%%%%%%%%%%%%%%%%%
%%%%%%%%%%%%%%%%%%%%%%%%%%%%%%%%%%%%%%%%%%%%%%%%%%
%%%%%%%%%%%%%%%%%%%%%%%%%%%%%%%%%%%%%%%%%%%%%%%%%%
%%%%%%%%%%%%%%%%%%%%%%%%%%%%%%%%%%%%%%%%%%%%%%%%%%
%%%%%%%%%%%%%%%%%%%%%%%%%%%%%%%%%%%%%%%%%%%%%%%%%%
%%%%%%%%%%%%%%%%%%%%%%%%%%%%%%%%%%%%%%%%%%%%%%%%%%
%%%%%%%%%%%%%%%%%%%%%%%%%%%%%%%%%%%%%%%%%%%%%%%%%%
%%%%%%%%%%%%%%%%%%%%%%%%%%%%%%%%%%%%%%%%%%%%%%%%%%
%%%%%%%%%%%%%%%%%%%%%%%%%%%%%%%%%%%%%%%%%%%%%%%%%%
%%%%%%%%%%%%%%%%%%%%%%%%%%%%%%%%%%%%%%%%%%%%%%%%%%
%%%%%%%%%%%%%%%%%%%%%%%%%%%%%%%%%%%%%%%%%%%%%%%%%%
%%%%%%%%%%%%%%%%%%%%%%%%%%%%%%%%%%%%%%%%%%%%%%%%%%
%%%%%%%%%%%%%%%%%%%%%%%%%%%%%%%%%%%%%%%%%%%%%%%%%%
%%%%%%%%%%%%%%%%%%%%%%%%%%%%%%%%%%%%%%%%%%%%%%%%%%
%%%%%%%%%%%%%%%%%%%%%%%%%%%%%%%%%%%%%%%%%%%%%%%%%%
%%%%%%%%%%%%%%%%%%%%%%%%%%%%%%%%%%%%%%%%%%%%%%%%%%
%%%%%%%%%%%%%%%%%%%%%%%%%%%%%%%%%%%%%%%%%%%%%%%%%%
%%%%%%%%%%%%%%%%%%%%%%%%%%%%%%%%%%%%%%%%%%%%%%%%%%
%%%%%%%%%%%%%%%%%%%%%%%%%%%%%%%%%%%%%%%%%%%%%%%%%%

\section{The Heckenberger--Kolb Calculi for the Irreducible Quantum Flag Manifolds} \label{section:DJ}

In this section we consider our motivating family of examples: the irreducible quantum flag manifolds endowed with their Heckenberger--Kolb calculi. We assume that the reader has some familiarity with the representation
theory of Lie algebras. For standard references see \cite{Humph, JantzenRepBook, VinbergOnishchik}.

\subsection{Drinfeld--Jimbo Quantum Groups}

In this subsection we recall the necessary definitions from the theory of Drinfeld--Jimbo quantum groups. We
refer the reader to~\cite{KSLeabh} for further details, as well to the seminal papers~\cite{DrinfeldICM, Jimbo1986}.
Let $\mathfrak{g}$ be a finite-dimensional complex semi\-simple Lie algebra of rank $r$. We fix a Cartan subalgebra $\mathfrak{h}$ with corresponding root system $\Delta \sseq \mathfrak{h}^*$, where $\mathfrak{h}^*$ denotes the linear dual of $\mathfrak{h}$.  
With respect to a choice of simple roots $\Pi = \{\alpha_1, \dots, \alpha_r\}$, denote by $(\cdot,\cdot)$ the symmetric bilinear form induced on $\mathfrak{h}^*$ by the  Killing form of $\mathfrak{g}$, normalised so that any shortest simple root $\alpha_i$ satisfies $(\alpha_i,\alpha_i) = 2$. Let $\{\varpi_1, \dots, \varpi_r\}$ denote the corresponding set of fundamental weights of~$\mathfrak{g}$. The {\em coroot} $\alpha_i^{\vee}$ of a simple root $\alpha_i$ is defined by
$
\alpha_i^{\vee} :=  2\alpha_i/(\alpha_i,\alpha_i).
$
The Cartan matrix $A = (a_{ij})_{ij}$ of $\mathfrak{g}$ is the $(r \times r)$-matrix defined by
$
a_{ij} := \big(\alpha_i^{\vee},\alpha_j\big).
$

Let  $q \in \bR$ such that  $q \notin \{ -1,0,1\}$, and denote $q_i := q^{(\alpha_i, \alpha_i)/2}$. The \emph{quantised enveloping algebra}  $U_q(\mathfrak{g})$ is the  noncommutative associative  algebra  generated by the elements   $E_i, F_i, K_i$, and  $K^{-1}_i$, for $ i=1, \ldots, r$,  subject to the relations 
\begin{align*}
 K_iE_j =  q_i^{a_{ij}} E_j K_i, \quad  K_iF_j= q_i^{-a_{ij}} F_j K_i, \quad  K_i K_j = K_j K_i, \quad K_iK_i^{-1} = K_i^{-1}K_i = 1,\\
  E_iF_j - F_jE_i  = \delta_{ij}\frac{K_i - K\inv_{i}}{q_i-q_i^{-1}}, ~~~~~~~~~~~~~~~~~~~~~~~~~~~~~~~~~~~~~~~~~
\end{align*}
along with the \emph{quantum Serre relations}  which we omit.
A Hopf algebra structure is defined  on $U_q(\mathfrak{g})$ by
\begin{align*}
\DEL(K_i) &= K_i \oby K_i,\quad  \DEL(E_i) = E_i \oby K_i + 1 \oby E_i, \quad \DEL(F_i) = F_i \oby 1 + K_i\inv  \oby F_i,\\
& \qquad S(E_i) =  - E_iK_i\inv,    \quad S(F_i) =  -K_iF_i, \quad  S(K_i) = K_i\inv,  \\
&\qquad \qquad \qquad \e(E_i) = \e(F_i) = 0, ~~ \e(K_i) = 1.     
\end{align*}  
A Hopf $*$-algebra structure, called the \emph{compact real form} of $U_q(\mathfrak{g})$, is defined by
\begin{align*}
K^*_i : = K_i, & & E^*_i := K_i F_i, & &  F^*_i  :=  E_i K_i \inv. 
\end{align*} 

Let $\mathcal{P}$ be the weight lattice of~$\fg$, and $\mathcal{P}^+$ its set of dominant integral weights.  For each $\mu\in\mathcal{P}^+$ there exists an irreducible finite-dimensional $U_q(\mathfrak{g})$-module  $V_\mu$, uniquely defined by the existence of a vector $v_{\hw}\in V_\mu$, which we call a {\em highest weight vector},  satisfying
\[
  E_i \triangleright v_{\hw}=0,\qquad K_i \triangleright v_{\hw} = q^{(\alpha_i,\mu)} v_{\hw}, \qquad
  \text{for all $i=1,\ldots,r$.}
\]
Moreover, $v_{\hw}$ is unique up to scalar multiple. We call any finite direct sum of such \mbox{$U_q(\mathfrak{g})$-representations} a {\em type-$1$ representation}. A vector $v\in V_\mu$ is called a \emph{weight vector} of weight~$\mathrm{wt}(v) \in \mathcal{P}$ if
\begin{align}\label{eq:Kweight}
K_i \triangleright v = q^{(\alpha_i, \mathrm{wt}(v))} v, & & \textrm{ for all } i=1,\ldots,r.
\end{align}
Finally, we note that since $U_q(\mathfrak{g})$ has an invertible antipode, we have an equivalence between $_{U_q(\mathfrak{g})}\textrm{Mod}$, the category of  left $U_q(\mathfrak{g})$-modules, and $\textrm{Mod}_{U_q(\mathfrak{g})}$, the category of right $U_q(\mathfrak{g})$-modules, as induced by the antipode.

%%%%%%%%%%%%%%%%%%%%%%%%%%%%%%%%%%%%%%%%%%%
%%%%%%%%%%%%%%%%%%%%%%%%%%%%%%%%%%%%%%%%%%%
%%%%%%%%%%%%%%%%%%%%%%%%%%%%%%%%%%%%%%%%%%%
%%%%%%%%%%%%%%%%%%%%%%%%%%%%%%%%%%%%%%%%%%%
%%%%%%%%%%%%%%%%%%%%%%%%%%%%%%%%%%%%%%%%%%%

\subsection{Quantum Coordinate Algebras and Quantum Flag Manifolds} 
In this subsection we recall some necessary material about quantised coordinate algebras, for further details see \cite[\textsection6 and \textsection7]{KSLeabh}
and~\cite{FRT}.
Let $V$ be a finite-dimensional left $U_q(\mathfrak{g})$-module, $v \in V$, and $f \in V^*$, the $\mathbb{C}$-linear dual of $V$, endowed with its  right \mbox{$U_q(\mathfrak{g})$-module} structure. An important point to note is that, with respect to the equivalence of left and right \mbox{$U_q(\mathfrak{g})$-modules} discussed above, the left module corresponding to $V^*_{\mu}$ is isomorphic to $V_{-w_0(\mu)}$, where $w_0$ denotes the longest element in the Weyl group of $\mathfrak{g}$.

Consider the function  $c^{\textrm{\tiny $V$}}_{f,v}:U_q(\mathfrak{g}) \to \bC$ defined by $c^{\textrm{\tiny $V$}}_{f,v}(X) := f\big(X \triangleright v\big)$. The {\em coalgebra of matrix coefficients} of $V$ is the subspace
\begin{align*}
C(V) := \text{span}_{\mathbb{C}}\!\left\{ c^{\textrm{\tiny $V$}}_{f,v} \,| \, v \in V, \, f \in V^*\right\} \sseq U_q(\mathfrak{g})^*.
\end{align*}
A $U_q(\fg)$-bimodule structure on~$C(V)$ is given by
\begin{equation} \label{eq:Zact}
  (Y\triangleright c^{\textrm{\tiny $V$}}_{f,v} \triangleleft Z)(X) := f\left((ZXY)\triangleright v\right)
  = c^{\textrm{\tiny $V$}}_{f\triangleleft Z, Y \triangleright v}  (X).
\end{equation}
Let $U_q(\mathfrak{g})^\circ$ denote the Hopf dual of $U_q(\mathfrak{g})$. By construction $C(V) \sseq U_q(\mathfrak{g})^\circ$, and moreover that a Hopf subalgebra of $U_q(\mathfrak{g})^\circ$ is given by 
\begin{equation}\label{eq:PeterWeyl}
\OO_q(G) := \bigoplus_{\mu \in \mathcal{P}^+} C(V_{\mu}).
\end{equation} 
We call $\OO_q(G)$, endowed with the dual $*$-structure, the {\em quantum coordinate algebra of~$G$}, where~$G$ is the compact, simply-connected, simple Lie group  having~$\mathfrak{g}$ as its complexified Lie algebra.

\subsection{Quantum Flag Manifolds} \label{subsection:QFM}

For $\{\alpha_i\}_{i\in S}$ a subset of simple roots,  consider the Hopf $*$-subalgebra
\begin{align*}
U_q(\mathfrak{l}_S) := \big< K_i, E_j, F_j \,|\, i = 1, \ldots, r; j \in S \big>,
\end{align*} 
which $q$-deforms the classical reductive Lie algebra $\mathfrak{l}_S$. We have obvious analogues of weight vectors, highest weight vectors, type-$1$ representations. The category of finite-dimensional type-$1$ modules is semisimple, and each simple object admits a highest weight vector unique to scalar multiple. Moreover, the weight of the highest weight vector determines the module up to isomorphism, giving a bijective correspondence between irreducible type-$1$ modules and the set of weights
\begin{align*}
\mathcal{P}^+ \cup \mathcal{P}_{S^c}, & & \textrm{ where } \mathcal{P}_{S^c} = \mathrm{span}_{\mathbb{Z}}\{\varpi_x \,|\, x \in \Pi\backslash S\}.
\end{align*}
Just as for $\OO_q(G)$, we can construct the type-$1$ dual of $U_q(\mathfrak{l}_S)$ using matrix coefficients and we denote this Hopf algebra by $\OO_q(L_S)$.

Restriction of domains gives us a surjective Hopf $*$-algebra map 
\begin{align*}
\pi_S :\OO_q(G) \to \OO_q(L_S),
\end{align*}
dual to the inclusion of $U_q(\mathfrak{l}_S)$ of $U_q(\mathfrak{g})$ (see also  Appendix \ref{app:Levi}). The {\em quantum flag manifold associated} to $S$ is the CQGA-homogeneous space 
\begin{align*}
\OO_q\big(G/L_S\big) := \OO_q \big(G\big)^{\text{co}\left(\OO_q(L_S)\right)}
\end{align*} 
associated to $\pi_S$. Takeuchi's equivalence now implies that every simple relative Hopf module over $\OO_q(G/L_S)$ is of the form
\begin{align*}
    \F_{\mu} := \OO_q(G) \square_{\OO_q(L_S)} W_{\mu}, & & \textrm{ for } \mu \in \mathcal{P}^+ \cup \mathcal{P}_{S^c},
\end{align*}
where by abuse of notation $W_{\mu}$ denotes the $\OO_q(L_S)$-comodule corresponding to $\mu$.

\subsection{Irreducible Quantum Flag Manifolds}

Let $S = \{\alpha_1, \dots, \alpha_r \} \setminus \{\alpha_x\}$ where $\alpha_x$ has coefficient $1$ in the expansion of the
highest root of $\mathfrak{g}$. Then we say that the associated quantum flag manifold is  {\em irreducible}. In
the classical limit of $q=1$, these homogeneous spaces reduce to the family of compact Hermitian symmetric
spaces, as classified, for example, in~\cite{BastonEastwood}. Presented in Table~\ref{table:CQFMs} of Appendix~\ref{app:B}  is a useful diagrammatic presentation of the set of simple roots defining the
irreducible quantum flag manifolds, where the node corresponding to $\alpha_x$ is the coloured one.   %From now on, we will restrict to the irreducible setting.

Let us now choose for once and for all, for each irreducible $U_q(\mathfrak{g})$-module $V_{\mu}$, a weight basis $\{v_i\}_{i=1}^{N}$, with corresponding dual basis $\{f_i\}_{i=1}^{N}$,  where $N := \dim(V_{\mu})$. As shown in \cite[Proposition 3.2]{HK}, for the irreducible case, a set of generators for $\OO_q(G/L_S)$ is given by 
\begin{align*}
z_{ij} := c^{V_{\mu_x}}_{f_i,v_N}c^{V_{-w_0(\mu_x})}_{v_j,f_N}  & & \text{ for } i,j = 1, \dots, N = \dim(V_{\mu_x}),
\end{align*}
where $v_N$ is the highest weight basis element of $V_{\mu_x}$, and $f_N$ is the lowest weight basis element of $V_{-w_0( \mu_x)}$.

In the classical case $\fl_S$ admits a direct sum decomposition $\fl_S = \fl_S^s \oplus \fc$ into the semisimple
part~$\fl_S^{\,\mathrm{s}}$ and the centre~$\fc$, given a decomposition 
$$
U(\mathfrak{l}_S) \simeq U(\mathfrak{l}^{\mathrm{\,s}}_S) \otimes U(\mathfrak{c}),
$$
where  $U(\mathfrak{c})$ is a commutative and cocommutative Hopf algebra generated by the distinguished sum of Chevalley generators $H_i \in \mathfrak{h}$,
\begin{align*}
H_{\varpi_x} := (A^{-1})_{x1}H_1 + \cdots + (A^{-1})_{xr} H_r,
\end{align*}
with  $A^{-1}$ is the inverse of the Cartan matrix of $\mathfrak{g}$. Now for $U_q(\mathfrak{g})$ the element
\begin{align*}
  Z  := K_{\det(A) \varpi_x} = K_1^{a_1} \cdots  K_r^{a_r}, & & \textrm { where } \det(A)\varpi_x =: a_1\alpha_1 + \ldots + a_r \alpha_r
\end{align*}
belongs to the centre of $U_q(\fl_S)$. Note that the scaling $\det(A)$ is chosen to ensure that $\det(A) \varpi_x$ is an element of the root lattice of $\mathfrak{g}$ (see also Remark \ref{rem:Kvarpi} below). The elements $Z$ and $Z^{-1}$ generate a commutative and cocommutative Hopf subalgebra, and by Schur's lemma $Z$ acts as a scalar multiple of the identity on any irreducible $U_q(\mathfrak{l}^{\,\mathrm{s}}_S)$-module. It is instructive to note that any $U_q(\mathfrak{l}_S)$-module is completely determined by its $U_q(\mathfrak{l}^{\,\mathrm{s}}_S)$-module structure and the action of the central element $Z$.

\begin{eg} 
For the case of the quantum projective plane $\OO_q(\mathbb{CP}^2)$, taking $S = \{\alpha_2\}$, which is to say crossing the first node of the Dynkin diagram, the central element is 
$$
Z = K_1^2K_2^1.
$$
Alternatively, crossing the last node, the central element is given by 
$$
Z = K_1^1K_2^2.
$$
For the quantum Lagrangian Grassmannian $\OO_q(\mathbf{L}_3) = \OO_q(\mathrm{Sp}_6/L_S)$, we have
$$
Z = K_1^1K_2^2 K_3^3.
$$
We note that in all three cases, the algebra generated by $U_q(\mathfrak{l}^{\,\,\mathrm{s}}_S)$ and the central element $Z$ is a proper subalgebra of $U_q(\mathfrak{l}_S)$.
\end{eg}

\begin{remark} \label{rem:Kvarpi}
One can make sense of the symbol $K_{\varpi_x}$ as an element of $U_q(\mathfrak{g},\mathcal{P})$, the extension of $U_q(\mathfrak{g})$ which includes generators of the form $K_{\lambda}$, for all $\lambda \in \mathcal{P}$. This allows one to avoid multiplication by the determinant of the Cartan matrix in the definition of the central element $Z$.  See for example \cite[\textsection 3]{SISSACPn}, which deals with the special case of quantum projective space, and \cite[\textsection 2D]{MTSUK} which deals with the general case.
\end{remark}

\subsection{Relative Line Modules over the Irreducible Quantum Flag Manifolds} \label{subsection:LBs}

In this subsection, we recall the necessary facts about the relative line modules over the irreducible quantum flag manifolds $\OO_q(G/L_S)$. We first observe that the one-dimensional $U_q(\mathfrak{l}_S)$-modules correspond to the weights in $\mathcal{P}_{S^c}$, which in turn implies that the one-dimensional $\OO_q(L_S)$-comodules also correspond to the weights in $\mathcal{P}_{S^c}$. Thus by Takeuchi's equivalence the relative line modules are indexed by the weights $\mathcal{P}_{S^c}$. For the special case of the irreducible quantum flag manifolds
$$
\mathcal{P}_{S^c} = \mathbb{Z}\varpi_x.
$$
In this case we denote by $\EE_{l}$ the relative line module corresponding to the weight $l\varpi_x$.

We make two important observations about relative line modules over the irreducible quantum flag manifolds, considered as submodules of $\OO_q(G)$: Firstly, we note that for all $l \in \mathbb{Z}$, we have $(\EE_l)^\ast \simeq \EE_{-l}$. Secondly, we note that for each $l \in \mathbb{Z}$, 
\begin{align*}
h:\EE_l \times \EE_l \to \OO_q(G/L_S), & & (e_1,e_2) \mapsto e_1e_2^*,
\end{align*}
gives $\EE_l$ the structure of an Hermitian relative line module. Since $\EE_l$ is a simple object in $^{~~~~\OO_q(G)}_{\OO_q(G/L_S)}\textrm{mod}_0$,  we see that $h$ is the unique such structure up to positive scalar multiple.

\subsection{The Heckenberger--Kolb Calculi} \label{subsect:hk}

The irreducible quantum flag manifolds are distinguished by the existence of an essentially unique $q$-deformation of their classical de Rham complex. The existence of such a canonical deformation is one of the most important results in the study of the noncommutative geometry of quantum groups, serving as a solid base from which to investigate more general classes of quantum spaces.  The following theorem is a direct consequence of results established in  \cite{HK}, \cite{HKdR}, and \cite{MarcoConj}.

\begin{thm}\label{thm:HKClass}
Over any irreducible quantum flag manifold $\OO_q(G/L_S)$, there exists a unique finite-dimensional left $\OO_q(G)$-covariant differential $*$-calculus
$$
\Omega^{\bullet}_q(G/L_S) \in \, ^{~~~\OO_q(G)}_{\OO_q(G/L_S)}\mathrm{mod}_0,
$$
of classical dimension, that is to say,
  \begin{align*}
    \dim \Phi\!\left(\Omega^{k}_q(G/L_S)\right) = \binom{2M}{k}, & & \text{ for all \,} k = 0, \dots, 2 M,
  \end{align*}
  where $M$ is the complex dimension of the corresponding classical manifold, as presented in
  Table~\ref{table:CQFMsEk} of Appendix~\ref{app:B}.  
\end{thm}

The calculus $\Omega^{\bullet}_q(G/L_S)$, which we call the \emph{Heckenberger--Kolb calculus} of $\OO_q(G/L_S)$,  has many remarkable properties. We begin with the existence of a unique covariant complex structure, following from the results of  \cite{HK}, \cite{HKdR}, and \cite{MarcoConj}.

\begin{prop} \label{prop:complexstructure}
Let $\OO_q(G/L_S)$ be an irreducible quantum flag manifold, and $\Omega^{\bullet}_q(G/L_S)$ its Heckenberger--Kolb differential $*$-calculus. Then the following hold:
\begin{enumerate}
\item $\Omega^{\bullet}_q(G/L_S)$ admits a unique left $\OO_q(G)$-covariant complex structure,  
$$
\Omega^{\bullet}_q(G/L_S) \simeq  \bigoplus_{(a,b)\in\mathbb{Z}_{\geq 0}^2} \Omega^{(a,b)} =: \Omega^{(\bullet,\bullet)},
$$ 
\item $\Omega^{(\bullet,\bullet)}$ is factorisable,
\item $\Omega^{(1,0)}$ and $\Omega^{(0,1)}$ are irreducible as objects in $^{~~~\OO_q(G)}_{\OO_q(G/L_S)}\mathrm{mod}_0$.
\end{enumerate}
\end{prop}

As observed in \cite[\textsection10.8]{MMF3} (using the same argument as presented in part (i) of Lemma \ref{lem:thethm}) there exists a real left $\OO_q(G)$-coinvariant form $\kappa \in \Omega^{(1,1)}$, and it is unique up to real scalar multiple. Moreover, by extending the representation theoretic argument given in \cite[\textsection4.4]{MMF3} for the case $\OO_q(\mathbb{CP}^n)$, the form $\kappa$ is readily seen to be a closed central element of $\Omega^{\bullet}_q(G/L_S)$. This motivated \cite[Conjecture~4.25]{MMF3}, where it was proposed that the pair $(\Omega^{(\bullet,\bullet)},\kappa)$ is a K\"ahler structure for the calculus. With suitable restrictions on the values of $q$, the conjecture was verified by Matassa in \cite[Theorem~5.10]{MarcoConj}.

\begin{thm}  \label{thm:MatassaKahler}
Let $\Om^\bullet_q(G/L_S)$ be the Heckenberger--Kolb calculus of the irreducible quantum flag manifold $\mathcal{O}_q(G/L_S)$. The pair $(\Om^{(\bullet,\bullet)},\kappa)$ is a covariant K\"ahler structure for all \mbox{$q \in \mathbb{R}_{>0}\setminus F$,} where $F$ is a finite, possibly empty, subset of $\mathbb{R}_{>0}$. Moreover, any element of $F$ is necessarily  non-transcendental.
\end{thm}

The question of when this K\"ahler structure gives a CQH-K\"ahler space was addressed \cite[Theorem 6.1]{CQHKS}. Taken together with irreducibility of the holomorphic forms (as recalled in part (iii) of Proposition \ref{prop:complexstructure} above), this gives us the following theorem. (Note that in the theorem we need to restrict to a strictly positive cone of real $\OO_q(G)$-coinvariant $(1,1)$-forms to get positivity, and we denote an arbitrary K\"ahler form in this cone by $\kappa_+$.)

\begin{thm} \label{thm:PosDefKappa}
For each irreducible quantum flag manifold $\mathcal{O}_q(G/L_S)$, there exists a real left $\OO_q(G)$-coinvariant $(1,1)$-form $\kappa_+$, uniquely defined up to strictly positive real multiple, such that a  CQH-K\"ahler space is given by the quadruple
$$
\mathbf{K}_S := \left(\OO_q(G/L_S), \, \Omega^{\bullet}_q(G/L_S), \,\Omega^{(\bullet,\bullet)}, \kappa_+ \right)\!,
$$
for all $q \in I$, where $I \sseq  \mathbb{R}_{>0}$ is a sufficiently small open interval around $1$.
\end{thm}

In the rest of this section we build upon this result, using it to apply the general framework of the paper to the study of the irreducible quantum flag manifolds.

\subsection{Positive and Negative Line Modules over \texorpdfstring{$\OO_q(G/L_S)$}{Oq(G/LS}} \label{section:BWPosit}

It was shown in~\cite[Theorem 4.5]{DOKSS} that each relative Hopf module over any irreducible quantum flag manifold $\OO_q(G/L_S)$ admits a unique relative holomorphic structure. Let us now recall the precise statement of this result for the special case of line modules.

\begin{thm}\label{thm:holomorLNES}
For every relative line module $\mathcal{E}_k$ over $\mathcal{O}_q(G/L_S)$, there exists a unique covariant $(0,1)$-connection 
\begin{align*}
\overline{\partial}_{\mathcal{E}_k}: \mathcal{E}_k \rightarrow \Omega^{(0,1)}_q(G/L_S) \otimes_{\mathcal{O}_q(G/L_S)} \mathcal{E}_k.
\end{align*}
Moreover, $\overline{\partial}_{\mathcal{E}_k}$ is flat, which is to say, it is a holomorphic structure.
\end{thm}

It now follows directly from Theorem~\ref{thm:threeway} that the relative line modules over the irreducible quantum flag manifolds $\OO_q(G/L_S)$ are either positive, flat, or negative. To differentiate between these possibilities, we will use the cohomological information given by the noncommutative Borel--Weil theorem for $\OO_q(G/L_S)$ established  in \cite{IrrBW, KMOS}.

\begin{thm} \label{thm:BorelWeil}
For any irreducible quantum flag manifold $\OO_q(G/L_S)$, it holds that 
\begin{enumerate}
 \item $H^0(\mathcal{E}_k)$ is an irreducible $U_q(\mathfrak{g})$-module of highest weight $k\varpi_{x}$, for all $k \in \mathbb{Z}_{\geq 0}$,
 \item $H^0(\mathcal{E}_{-k}) = 0$, for all $k \in \mathbb{Z}_{>0}$,
\end{enumerate}
where $\Pi\backslash S$ consists of the simple root $\alpha_x$.
\end{thm}

Using this cohomological information, we can now apply Theorem~\ref{thm:threeway} to determine which modules are positive and which are negative.

\begin{thm} \label{thm:posnegLINEs}   For any irreducible quantum flag manifold $\OO_q(G/L_S)$, and any  $k \in \mathbb{Z}_{>0}$, 
\begin{enumerate}
\item $\mathcal{E}_k > 0$,
\item $\mathcal{E}_{-k} <0$,
\end{enumerate}
for any  K\"ahler structure identified in Theorem \ref{thm:PosDefKappa}.
\end{thm}

\begin{eg} \label{PodlesCurvature}
For any  differential manifold the curvature of a connection is additive over tensor products of vector bundles. In particular, any tensor power of a positive line bundle over a complex manifold is again positive. In the noncommutative setting tensoring two holomorphic modules is more problematic, and one needs to consider bimodule connections (in the sense of \cite{DVMadoreMouradBimodule,MadoreBimodule,DVMichor,BimoduleConn}). Even in this case, curvature does not behave additively. In particular, for a bimodule holomorphic line module $\EE$, we cannot directly conclude positivity of $\EE^{\otimes_B k}$ from positivity of $\EE$, making the general approach of this paper all the more valuable.

For the case of the standard Podle\'s sphere $\OO_q(S^2)$ the curvature can be explicitly calculated using the theory of quantum principal bundles, see \cite[Example 5.23]{BeggsMajid:Leabh} for details. In the conventions of this paper, for any line module $\EE_k$ over $\mathcal{O}_q(S^2)$, it holds that 
\begin{align*}
\nabla^2(e) =  -(k)_{q^{-2}} \mathbf{i} \kappa \otimes e, & & \textrm{ for all } e \in \mathcal{E}_k,
\end{align*}
where the quantum integer is given explicitly by
$
(k)_{q^{-2}} := 1 + q^{-2} + q^{-4} + \cdots + q^{-2(k-1)},
$
and we have chosen the unique K\"ahler form $\kappa$ satisfying
\begin{align} \label{eqn:firstCPNChern}
\nabla^2(e) = - \mathbf{i} \kappa \otimes e, & & \text{for all }   e \in \EE_1.
\end{align}
This suggests that there is some type of $q$-deformed (or braided) additivity underlying these results. Understanding this process presents itself as an interesting and important future goal.
\end{eg}

%%%%%%%%%%%%%%%%%%%%%%%%%%%%%%%%%%%%
%%%%%%%%%%%%%%%%%%%%%%%%%%%%%%%%%%%%
%%%%%%%%%%%%%%%%%%%%%%%%%%%%%%%%%%%%
%%%%%%%%%%%%%%%%%%%%%%%%%%%%%%%%%%%%
%%%%%%%%%%%%%%%%%%%%%%%%%%%%%%%%%%%%
%%%%%%%%%%%%%%%%%%%%%%%%%%%%%%%%%%%%

\subsection{A Fano Structure for the Irreducible Quantum Flag Manifolds} \label{section:HKFano}

Building on the definition of K\"ahler structure, the notion of a noncommutative Fano structure was introduced in \cite[\textsection 8.2]{OSV}, directly generalising the classical definition of a Fano manifold. Moreover, for any CQH-Fano space $\mathbf{F}$ (defined in the obvious way as a refinement of the definition of a CQH-K\"ahler structure) the noncommutative Kodaira vanishing theorem established in \cite[Theorem 8.3]{OSV} implies that the antiholomorphic cohomology $H^{(0,\bullet)}(\EE)$ of any  positive line module $\EE$ over $\mathbf{F}$ will be concentrated in degree zero. 

\begin{defn} \label{def:fano}
A {\em  Fano structure} for a differential $*$-calculus $\Omega^{\bullet}$, of total degree $2n$, is a factorisable K\"ahler structure $(\Omega^{(\bullet,\bullet)},\kappa)$ such that the  holomorphic Hermitian module
$
\big(\Omega^{(n,0)}, g_{\kappa}, \adel \big)
$
is negative.
\end{defn}

We now verify the Fano condition for the irreducible quantum flag manifolds. Besides being an interesting result in its own right, it is also a necessary step for the proof of the Bott--Borel--Weil theorem for the irreducible quantum flag manifolds \cite[\textsection 9]{OSV}. We note first that since a line module $\EE_l$ is negative if and only if $l$ is a negative integer, verifying the Fano condition amounts to showing that $\Omega^{(M,0)} \simeq \EE_{-k}$, for some $k > 0$, where we recall that $M$ denotes the complex dimension of the corresponding classical manifold $G/L_S$. This we  do by producing a general description of $k$ in terms of the Cartan matrix of $\mathfrak{g}$, for the special case of $q 
\in I$, where $I$ is as defined in Theorem \ref{thm:PosDefKappa}. (In Table~\ref{table:CQFMsEk} of Appendix~\ref{app:B}, we present the explicit values of $k$ for each series of the irreducible quantum flag manifolds.)  

\begin{thm} \label{thm:HKFano} Let $\OO_q(G/L_S)$ be an irreducible quantum flag manifold, endowed with its Heckenberger--Kolb calculus $\Omega_q(G/L_S)$. For any
$q \in I$, the pair $(\Omega^{(\bullet,\bullet)},\kappa_+)$ is a Fano structure.
\end{thm}
\begin{proof}
Since the complex structure $\Omega^{(\bullet, \bullet)}$ is factorisable, we only need to verify condition (ii) of Definition~\ref{def:fano}, that is, show that  $(\Omega^{(M,0)}, g_{\kappa_+}, \adel)$ is a negative line module, where $2M$ is the total degree of $\Omega_q^\bullet(G/L_S)$. First, we will identify the unique $k \in\mathbb{Z}_{>0}$ such that 
\begin{align} \label{eqn:antiHoloEE}
  \Phi\big(\Omega^{(M,0)}\big) \simeq \Phi(\EE_{-k}).
  %= \left(\Phi(\EE_{-1})\right)^{\otimes k}\!.
\end{align}
To do so, we will compare the actions of the central element $Z$ on the \mbox{$U_q(\mathfrak{l}_S)$-modules} $\Phi(\EE_{-k})$ and $\Phi(\Omega^{(M,0)})$. Now since $\Phi(\Omega^{(1,0)})$ is  irreducible as a $U_q(\mathfrak{l}_S)$-module, $Z$ acts on $\Phi(\Omega^{(1,0)})$ as multiplication by some scalar~$\gamma$.
For any $z_{ij} \in \OO_q(G/L_S)$ we see that 
\begin{align} \label{eqn:Zadelzij}
[\adel z_{ij}] \triangleleft Z = \left[\del(c^{\varpi_x}_{f_i\triangleleft Z, v_N}c^{-w_0(\varpi_x)}_{v_j \triangleleft Z,f_N})\right] = q^{-(\varpi_x,\mathrm{wt}(f_i) + \mathrm{wt}(v_j))\det(A)}[\adel z_{ij}].
\end{align}
As shown in~\cite[Proposition~3.6]{HKdR}, a basis of~$\Phi(\Omega^{(1,0)})$ is given by 
$$
\left\{[\del z_{iN}] \,| \, \text{ for } i\in J_{(1)}\right\}\!,
$$
where $J_{(1)}$ is the subset of the index set $J:= \{1,\ldots,\dim(V_{\varpi_x})\}$ defined by
$$
J_{(1)} := \{ i\in J \mid (\varpi_x, \mathrm{wt}(v_i))  = (\varpi_x,\varpi_x -\alpha_x)\}.
$$ 
Now for any $i\in J_{(1)}$, we have
\begin{align*}
    (\varpi_x,\mathrm{wt}(f_i) + \mathrm{wt}(v_N)) = (\varpi_x, - \varpi_x + \alpha_x) + (\varpi_x,\varpi_x) = (\varpi_x, \alpha_x).
\end{align*}
Thus we see that $\gamma = q^{-(\varpi_x,\alpha_x)\det(A)}$.
Since the category of $U_q(\mathfrak{l}_S)$-modules is semisimple, the projection from $\Phi(\Omega^{(1,0)})^{\otimes M}$ to $\Phi(\Omega^{(M,0)})$ splits, meaning that $\Phi(\Omega^{(M,0)})$ can be considered as a submodule of $\Phi(\Omega^{(1,0)})^{\otimes M}$. Thus we see that $Z$ must act on~$\Phi(\Omega^{(M,0)})$ as multiplication by $q^{-M (\varpi_x,\alpha_x)\det(A)}$.

The definition of $\EE_k = \Psi(W_{k\varpi_x})$ implies that $Z$ acts on $\Phi(\EE_k)$ as the scalar $q^{k(\varpi_x,\varpi_x)\det(A)}$. Thus it follows from \eqref{eqn:antiHoloEE} that $k$ is uniquely determined by 
$$
 M (\varpi_x,\alpha_x) \det(A)  =  k (\varpi_x,\varpi_x)\det(A).
$$
Recalling the standard identity
$$
(\varpi_x,\varpi_x) =  \frac{(\alpha_x,\alpha_x)}{2}(A^{-1})_{xx}, 
$$
where $(A^{-1})_{xx}$ denotes the $x$-diagonal entry of the inverse of the Cartan matrix, we see
\begin{align*} 
k = \frac{ M(\varpi_x,\alpha_x)}{(\varpi_x,\varpi_x)} = \frac{ 2M(\varpi_x,\alpha_x)}{(\alpha_x,\alpha_x)(A^{-1})_{xx}} = \frac{ M(\varpi_x,\alpha_x^{\vee})}{(A^{-1})_{xx}} = \frac{M}{(A^{-1})_{xx}}.
\end{align*}
(We note that since $k$ is by definition an integer, the rational number $M/(A^{-1})_{xx}$ is necessarily an integer, as can be confirmed by direct investigation. See Table~\ref{table:CQFMsEk} of Appendix~\ref{app:B} for explicit values.) It remains to show that $k > 0$. Indeed, each entry of the matrix $(A^{-1})_{xx}$ is a positive rational number (see, for example, \cite[Table~2]{VinbergOnishchik} for an explicit presentation of the values). Thus, for every irreducible quantum flag manifold $\OO_q(G/L_S)$, the scalar $k$ must be an element of $\mathbb{Z}_{>0}$.
\end{proof}

\appendix

\section{Some Categorical Equivalences} \label{app:A}

In this appendix we present a number of categorical equivalences, all ultimately derived from Takeuchi's equivalence \cite{Tak}. These equivalences play a prominent role in the paper, giving us a formal framework in which to understand covariant differential calculi as relative Hopf modules.

\subsection{Takeuchi's  Equivalence} \label{app:TAK}

This subsection concisely presents Takeuchi's equivalence for relative Hopf modules, as originally established in \cite{Tak}, and developed in \cite{WignerMas}. We also present the monoidal version, as considered in \cite[\textsection 4]{MMF2}, and then restrict to the finitely generated subcategory of relative Hopf modules, as considered, for example, in \cite[Corollary 2.5]{MMF3}.

Let $\pi : A \to H$ be a surjective Hopf algebra map between Hopf algebras $A$ and $H$. Then a \emph{homogeneous right $H$-coaction} is given by the map 
$
\Delta_R := (\id \otimes \pi) \circ \Delta : A \to A \otimes H.
$
The associated \emph{quantum homogeneous space} is defined to be the space of coinvariant elements, $A^{\co(H)}$, that is, \[ A^{\co(H)} := \{ a \in A \mid \Delta_R (a) = a \otimes 1 \}.\]

For any quantum homogeneous space $B = A^{\co(H)}$, we define  $^A_B\mathrm{Mod}_B$ to be the category whose  objects are left \mbox{$A$-comodules} \mbox{$\DEL_L:\mathcal{F} \to A \otimes \mathcal{F}$}, endowed with a $B$-bimodule structure such that $\DEL_L(bfc) = \Delta_L(b)\DEL_L(f)\Delta(c)$,  for all  $f \in \mathcal{F}, \, b,c \in B$, and whose morphisms  are left $A$-comodule, $B$-bimodule, maps.

Let $^H\!\mathrm{Mod}_B$ denote the category of \emph{relative Hopf modules}, that is, the category whose objects are left $H$-comodules $\Delta_L: V \to H \otimes V$, endowed with a right $B$-module structure, such that $\Delta_L(vb) = v_{(-1)}\pi(b_{(1)}) \otimes v_{(0)}b_{(2)}$, for all $v \in V, \, b \in B$, and whose morphisms are left $H$-comodule maps and right $B$-module maps.

Consider the functor $\Phi:{}^A_B\mathrm{Mod}_B \to {}^{H}\mathrm{Mod}_B$, given by $\Phi(\mathcal{F}) :=   \mathcal{F}/B^+\mathcal{F}$, where the left $H$-comodule structure of $\Phi(\mathcal{\F})$ is given by 
$
\Delta_L[f] := \pi(f_{(-1)})\otimes [f_{(0)}],
$
with square brackets denoting the coset of an element in $\Phi(\mathcal{\F})$. In the other direction, we define a functor $\Psi: {}^{H}\mathrm{Mod}_B \to {}^A_B\mathrm{Mod}_B$ by setting $\Psi(V) := A \,\square_{H} V$, where the left $A$-comodule structures of $\Psi(V)$ is defined on the first tensor factor, the right $B$-module structure is the diagonal one, and if $\gamma$ is a morphism in ${}^{H}\mathrm{mod}_B$, then $\Psi(\gamma) := \id \otimes \gamma$. %Note that $A$ is naturally an object in ${}^A_B\mathrm{mod}$, and that $\Phi(A) = \pi_P(A).$

An adjoint equivalence of categories between~${}^A_B\mathrm{Mod}_B$ and~${}^{H}\mathrm{Mod}_B$, which we call \emph{Takeuchi's equivalence}, is given by the functors $\Phi$ and $\Psi$, and the unit natural isomorphism
\begin{align*}
\unit: \F \to \Psi \circ \Phi(\F), & & f \mapsto f_{(-1)} \otimes [f_{(0)}].
\end{align*}
The \emph{dimension} $\mathrm{dim}(\F)$ of an object $\F \in {}^A_B\mathrm{Mod}$ is the vector space dimension of $\Phi(\F)$. As observed in \cite[Corollary 2.7]{MMF2}, the inverse of the unit $\unit$ of the equivalence admits a useful explicit description: 
\begin{align} \label{eqn:unitinverse} 
\unit^{-1}\!\left(\sum_i f_i \otimes [g_i] \right) = \sum_i f_iS\big((g_i)_{(-1)}\big)(g_i)_{(0)}.
\end{align}

Consider $^A_B\textrm{Mod}_0$ the full subcategory of ${}^A_B\mathrm{Mod}_B$ whose objects $\F$ satisfy $B^+\F = \F B^+$. The corresponding full subcategory ${}^H\mathrm{Mod}_0$ of ${}^H\mathrm{Mod}_B$ is given by objects with the trivial right $B$-action. The category $^A_B\textrm{Mod}_0$ comes equipped with a monoidal structure given by the tensor product $\otimes_B$. Moreover, with respect to the obvious monoidal structure on ${}^H\mathrm{Mod}_0$, Takeuchi's equivalence is readily endowed with the structure of a monoidal equivalence (see \cite[\textsection 4]{MMF2}). An immediate implication is that an object $\EE$ is invertible (that is, it is a relative line module) if and only if $\dim(\EE) = 1$.

Finally, we consider $^A_B\textrm{mod}_0$ the full subcategory of ${}^A_B\mathrm{Mod}_0$ whose objects are finitely generated as left $B$-modules, and note that it is a monoidal subcategory of ${}^A_B\mathrm{Mod}_0$. The corresponding full subcategory ${}^H\mathrm{mod}$ of ${}^H\mathrm{Mod}$ has as objects the finite-dimensional left $H$-comodules, and its monoidal structure is the usual tensor product of comodules.

\subsection{Conjugates and Duals}

Let us now assume that $A$ and $H$ are Hopf $*$-algebras, and that $\pi:A \to H$ is a Hopf $*$-algebra map. For any relative Hopf module $\F$, let $\overline{\F}$ be the relative Hopf module whose bimodule structure is defined by  $b\overline{f}c = \overline{c^*fb^*}$, and whose left $A$-comodule  structure is defined by $\Delta_L(\overline{f}) := (f_{(-1)})^* \otimes \overline{f_{(0)}}$. Note that as a right $B$-module, $\overline{\F}$ is isomorphic to the conjugate of $\F$ as defined in \ref{subsection:HVBS}, justifying the choice of notation. As shown in \cite[Corollary 2.11]{OSV},  if  $\F \in \, ^A_B\textrm{Mod}_0$, then $\overline{\F} \in \, ^A_B\textrm{Mod}_0$. It is instructive to note that  the corresponding operation, through Takeuchi's equivalence, on any object in $V \in \,^H\textrm{Mod}_0$ is the usual complex conjugate of  $V$ coming from the  Hopf $*$-algebra structure of $H$.

We now restrict to the categories $^A_B\mathrm{mod}_0$ and $^H\mathrm{mod}$, and discuss dual objects. Since $^H\mathrm{mod}$ is a rigid monoidal category, $^A_B\mathrm{mod}$, or equivalently $^A_B\mathrm{mod}_0$, is a rigid monoidal category. In particular, every object $\F \in \, ^A_B\mathrm{mod}_0$ admits a dual object, which we denote by ${}^{\vee}\!\F$. Moreover, as shown in  \cite[Appendix A]{OSV}, by extending its usual bimodule structure, we can give $\, _{B}\mathrm{Hom}(\F,B)$   the structure of an object in $\,^A_B\mathrm{Mod}_B$, with respect to which it is right dual to $\F$. Now $^A_B\mathrm{mod}_0$ is a monoidal subcategory of $\,^A_B\mathrm{Mod}_B$. Thus since right duals are unique up to unique isomorphism,  ${}^{\vee}\!\F$ must be isomorphic to $\, _{B}\mathrm{Hom}(\F,B)$, justifying the abuse of notation. Finally, we note that if $H$ is a CQGA, then for any $V \in \,^H\mathrm{mod}_0$, its dual and conjugate are always isomorphic, see \cite[Theorem 11.27]{KSLeabh} for details.

%%%%%%%%%%%%%%%%%%%%%%%%%%%%%%%%%%%%%%%
%%%%%%%%%%%%%%%%%%%%%%%%%%%%%%%%%%%%%%%
%%%%%%%%%%%%%%%%%%%%%%%%%%%%%%%%%%%%%%%
%%%%%%%%%%%%%%%%%%%%%%%%%%%%%%%%%%%%%%%
%%%%%%%%%%%%%%%%%%%%%%%%%%%%%%%%%%%%%%%
%%%%%%%%%%%%%%%%%%%%%%%%%%%%%%%%%%%%%%%
%%%%%%%%%%%%%%%%%%%%%%%%%%%%%%%%%%%%%%%
%%%%%%%%%%%%%%%%%%%%%%%%%%%%%%%%%%%%%%%
%%%%%%%%%%%%%%%%%%%%%%%%%%%%%%%%%%%%%%%
%%%%%%%%%%%%%%%%%%%%%%%%%%%%%%%%%%%%%%%
%%%%%%%%%%%%%%%%%%%%%%%%%%%%%%%%%%%%%%%
%%%%%%%%%%%%%%%%%%%%%%%%%%%%%%%%%%%%%%%

\section{A Remark on the Definition of the Quantum Levi Subgroup} \label{app:Levi}

In this subsection we show the image of the restriction map $\pi_S$ defined in \textsection \ref{subsection:QFM} is the type-$1$ dual of $U_q(\mathfrak{l}_S)$, for any subset of simple nodes of a semisimple Lie algebra $\mathfrak{g}$. In fact, this is taken as the definition of $\OO_q(L_S)$ in \cite{DijkStok}, and is the definition used in \cite{HK,HKdR,DOKSS}. This result is a basic exercise in representation theory, and can be directly concluded from \cite[Lemma 2.1]{HK:BGGdR}. We include a proof for the reader's convenience. 

\begin{prop}
Let $S$ be a subset of the simple roots $\Pi$ of $\mathfrak{g}$, then the Hopf $*$-algebra map $\pi_S:\OO_q(G) \to \OO_q(L_S)$ is surjective.
\end{prop}
\begin{proof}
Take an arbitrary finite-dimensional type-$1$ irreducible $U_q(\mathfrak{g})$-module $V$ and decompose it into irreducible type-$1$  $U_q(\mathfrak{l}_S)$-submodules:
$$
V \simeq \bigoplus_{\mu \, \in \mathcal{P}^+ \cup \mathcal{P}_{S^c}} W_{\mu}.
$$
This gives a corresponding coordinate coalgebra decomposition 
    \begin{align} \label{eqn:branching}
    \pi_S(C(V_{\lambda})) \simeq \bigoplus_{\mu \in \mathcal{P}^+ \cup \mathcal{P}_{S^c}} C(W_{\mu}).
    \end{align}
Thus we see that the image of $\pi_S$ is contained in $\OO_q(L_S)$, the type-$1$ dual of $U_q(\mathfrak{l}_S)$. 
    
To prove surjectivity we need to show that for every weight  $\nu \in \mathcal{P}^+ \cup \mathcal{P}_{S^c}$, the coordinate algebra $C(W_{\nu})$ is contained in the image of $\pi_S$. By \eqref{eqn:branching}, this would follow from a demonstration that every $W_{\nu}$ appears as a $U_q(\mathfrak{l}_S)$-submodule of some $U_q(\mathfrak{g})$-module. 
        
Every element of $\mathcal{P}^+ \cup \mathcal{P}_{S^c}$ is a sum of weights of the form $\lambda  + \mu$, where 
\begin{align*}
\lambda \in \mathcal{P}^+, & & \textrm{ and } & & -\mu \in \mathcal{P}_{S^c} \cap \mathcal{P}^+.
\end{align*} 
Choose a highest weight vector $v_{\hw}$ in the irreducible $U_q(\mathfrak{g})$-module $V_{\lambda}$, and choose a lowest weight vector $v_{\lw}$ in the irreducible $U_q(\mathfrak{g})$-module $V_{w_0(\mu)}$. We see that since
$$
v_{\hw} \otimes v_{\lw} \in V_{\lambda} \otimes V_{\mu}
$$
is a $U_q(\mathfrak{l}_S)$-highest weight vector, $U_q(\mathfrak{l}_S)v_{\hw} \otimes v_{\lw}$ is an irreducible  $U_q(\mathfrak{l}_S)$-submodule of $V_{\lambda} \otimes V_{w_0(\mu)}$ of highest weight $\lambda + \mu$. Thus we see that $C(W_{\lambda+\mu})$ is contained in the image of $\pi_S$, and hence that $\pi_S$ is surjective.
\end{proof}

%%%%%%%%%%%%%%%%%%%%%%%%%%%%%%%%%%%%%%%
%%%%%%%%%%%%%%%%%%%%%%%%%%%%%%%%%%%%%%%
%%%%%%%%%%%%%%%%%%%%%%%%%%%%%%%%%%%%%%%
%%%%%%%%%%%%%%%%%%%%%%%%%%%%%%%%%%%%%%%
%%%%%%%%%%%%%%%%%%%%%%%%%%%%%%%%%%%%%%%
%%%%%%%%%%%%%%%%%%%%%%%%%%%%%%%%%%%%%%%
%%%%%%%%%%%%%%%%%%%%%%%%%%%%%%%%%%%%%%%
%%%%%%%%%%%%%%%%%%%%%%%%%%%%%%%%%%%%%%%
%%%%%%%%%%%%%%%%%%%%%%%%%%%%%%%%%%%%%%%
%%%%%%%%%%%%%%%%%%%%%%%%%%%%%%%%%%%%%%%
%%%%%%%%%%%%%%%%%%%%%%%%%%%%%%%%%%%%%%%
%%%%%%%%%%%%%%%%%%%%%%%%%%%%%%%%%%%%%%%

\section{Tables for the Irreducible Quantum Flag Manifolds} \label{app:B}

We recall the standard pictorial description of the quantum Levi subalgebras defining the irreducible quantum flag manifolds, given in terms of Dynkin diagrams. 

\begin{center}
  \begin{table}[ht]
    \captionof{table}{\small{Irreducible Quantum Flag Manifolds: organised by series,  with defining  crossed node numbered according to \cite[\textsection11.4]{Humph}, CQGA-homogeneous space symbol and name}} \label{table:CQFMs}
{\small \renewcommand{\arraystretch}{2}%
 \begin{tabular}{|c|c|c|C{4cm}|}

\hline

\small $A_n$&
\begin{tikzpicture}[scale=.5]
\draw
(0,0) circle [radius=.25] 
(8,0) circle [radius=.25] 
(2,0)  circle [radius=.25]  
(6,0) circle [radius=.25] ; 

\draw[fill=black]
(4,0) circle  [radius=.25] ;

\draw[thick,dotted]
(2.25,0) -- (3.75,0)
(4.25,0) -- (5.75,0);

\draw[thick]
(.25,0) -- (1.75,0)
(6.25,0) -- (7.75,0);
\end{tikzpicture} & \small $\OO_q(\text{Gr}_{x,n+1})$ & \small quantum Grassmannian  \\

%%%%%%%%%%%%%%%%%%%%%%%%%%%%%%%%%%%%
\small $B_n$ &
\begin{tikzpicture}[scale=.5]
\draw
(4,0) circle [radius=.25] 
(2,0) circle [radius=.25] 
(6,0)  circle [radius=.25]  
(8,0) circle [radius=.25] ; 
\draw[fill=black]
(0,0) circle [radius=.25];

\draw[thick]
(.25,0) -- (1.75,0);

\draw[thick,dotted]
(2.25,0) -- (3.75,0)
(4.25,0) -- (5.75,0);

\draw[thick] 
(6.25,-.06) --++ (1.5,0)
(6.25,+.06) --++ (1.5,0);                      

\draw[thick]
(7,0.15) --++ (-60:.2)
(7,-.15) --++ (60:.2);
\end{tikzpicture} & \small $\OO_q(\mathbf{Q}_{2n+1})$ & \small {odd} quantum quadric  \\ 

%%%%%%%%%%%%%%%%%%%%%%%%%%%%%%%%%%%%
\small $C_n$& 
\begin{tikzpicture}[scale=.5]
\draw
(0,0) circle [radius=.25] 
(2,0) circle [radius=.25] 
(4,0)  circle [radius=.25]  
(6,0) circle [radius=.25] ; 
\draw[fill=black]
(8,0) circle [radius=.25];

\draw[thick]
(.25,0) -- (1.75,0);

\draw[thick,dotted]
(2.25,0) -- (3.75,0)
(4.25,0) -- (5.75,0);

\draw[thick] 
(6.25,-.06) --++ (1.5,0)
(6.25,+.06) --++ (1.5,0);                      

\draw[thick]
(7,0) --++ (60:.2)
(7,0) --++ (-60:.2);
\end{tikzpicture} &\small   $\OO_q(\mathbf{L}_{n})$ & \small %symmetric 
quantum Lagrangian Grassmannian    \\ 

%%%%%%%%%%%%%%%%%%%%%%%%%%%%%%%%%%%%
\small $D_n$& 
\begin{tikzpicture}[scale=.5]

\draw[fill=black]
(0,0) circle [radius=.25] ;

\draw
(2,0) circle [radius=.25] 
(4,0)  circle [radius=.25]  
(6,.5) circle [radius=.25] 
(6,-.5) circle [radius=.25];

\draw[thick]
(.25,0) -- (1.75,0)
(4.25,0.1) -- (5.75,.5)
(4.25,-0.1) -- (5.75,-.5);

\draw[thick,dotted]
(2.25,0) -- (3.75,0);
\end{tikzpicture} &\small   $\OO_q(\mathbf{Q}_{2n})$ & \small even quantum quadric  \\

%%%%%%%%%%%%%%%%%%%%%%%%%%%%%%%%%%%% 
\small $D_n$ & 
\begin{tikzpicture}[scale=.5]
\draw
(0,0) circle [radius=.25] 
(2,0) circle [radius=.25] 
(4,0)  circle [radius=.25] ;

\draw[fill=black] 
(6,.5) circle [radius=.25];
\draw%[fill=black] 
(6,-.5) circle [radius=.25];

\draw[thick]
(.25,0) -- (1.75,0)
(4.25,0.1) -- (5.75,.5)
(4.25,-0.1) -- (5.75,-.5);

\draw[thick,dotted]
(2.25,0) -- (3.75,0);
\end{tikzpicture} &\small   $\OO_q(\textbf{S}_{n})$ & \small quantum spinor variety   \\

%%%%%%%%%%%%%%%%%%%%%%%%%%%%%%%%%%%%
\small $E_6$& \begin{tikzpicture}[scale=.5]
\draw
(2,0) circle [radius=.25] 
(4,0) circle [radius=.25] 
(4,1) circle [radius=.25]
(6,0)  circle [radius=.25] ;

\draw%[fill=black] 
(0,0) circle [radius=.25];
\draw[fill=black] 
(8,0) circle [radius=.25];

\draw[thick]
(.25,0) -- (1.75,0)
(2.25,0) -- (3.75,0)
(4.25,0) -- (5.75,0)
(6.25,0) -- (7.75,0)
(4,.25) -- (4, .75);
\end{tikzpicture}

 &\small  $\OO_q(\mathbb{OP}^2)$ & \small  quantum Caley plane   \\
%%%%%%%%%%%%%%%%%%%%%%%%%%%%%%%%%%%%
\small $E_7$& 
\begin{tikzpicture}[scale=.5]
\draw
(0,0) circle [radius=.25] 
(2,0) circle [radius=.25] 
(4,0) circle [radius=.25] 
(4,1) circle [radius=.25]
(6,0)  circle [radius=.25] 
(8,0) circle [radius=.25];

\draw[fill=black] 
(10,0) circle [radius=.25];

\draw[thick]
(.25,0) -- (1.75,0)
(2.25,0) -- (3.75,0)
(4.25,0) -- (5.75,0)
(6.25,0) -- (7.75,0)
(8.25, 0) -- (9.75,0)
(4,.25) -- (4, .75);
\end{tikzpicture} &\small   $\OO_q(\textbf{F})$ 
& \small   quantum Freudenthal variety   \\
\hline
\end{tabular}
 }
\end{table}
\end{center}

For a diagram of rank~$r$, to the black node $\alpha_x$ we associate the set $S := \{\alpha_1,\ldots,\alpha_r\}\backslash\{\alpha_x\}$, with corresponding Levi subgroup $L_S$. The irreducible quantum flag manifold is then given by the coinvariant subspace $\mathcal{O}_q(G/L_S) \sseq \OO_q(G)$. Note that any automorphism of a Dynkin diagram results in an isomorphic quantum flag manifold, which is not denoted in the diagram. In particular, for the case of $D_n$ and $E_6$, colouring the second spinor node, and the first node, respectively, produces an isomorphic copy of the corresponding quantum flag manifold.

We present an explicit description of the canonical line modules (see Table \ref{table:CQFMsEk}) of the irreducible quantum flag
manifolds using the approach of Theorem~\ref{thm:HKFano}. All line modules are indexed by the negative integers, and hence are negative in the
sense of Definition~\ref{defn:positiveNegativeVB}. The values coincide with their classical counterparts, see for example, \cite[\S~II.4]{JantzenRepBook}. This allows us to conclude in Theorem~\ref{thm:HKFano} that the K\"ahler structure of each irreducible quantum flag manifold is of Fano type. 

\begin{center}
  \begin{table}[ht]
\captionof{table}{\small{Irreducible Quantum Flag Manifolds: CQGA-homogeneous space symbol, corresponding Heckenberger--Kolb calculus complex dimension, and the space of top holomorphic forms identified as a line module}} \label{table:CQFMsEk}
{\small \renewcommand{\arraystretch}{1.8}%

\begin{tabular}{|c|c|c| }
  \hline
~~  $\OO_q(G/L_S)$ ~~ & ~~ $M := \mathrm{dim}\!\left(\Omega^{(1,0)}\right)$ ~~& ~~Canonical line module $\Omega^{(M,0)}$  ~~ \\ 
  \hline
  $\OO_q(\text{Gr}_{s,n+1})$ & $s(n\!-\!s\!+\!1)$ & $\EE_{-(n+1)}$ 
 \\ 
 %%%%%%%%%%%%%%%%%%%%%%%%%%%%%%%%%%%% 
  $\OO_q(\mathbf{Q}_{2n+1})$ &  $2n-1$ & $\EE_{-2n+1}$  \\ 
  %%%%%%%%%%%%%%%%%%%%%%%%%%%%%%%%%%%%
  $\OO_q(\mathbf{L}_{n})$ &  $\frac{n(n+1)}{2}$  & $\EE_{-(n+1)}$ \\
%%%%%%%%%%%%%%%%%%%%%%%%%%%%%%%%%%%%
  $\OO_q(\mathbf{Q}_{2n})$ &  $2(n-1)$ & $\EE_{-2(n-1)}$\\ 
%%%%%%%%%%%%%%%%%%%%%%%%%%%%%%%%%%%%
  $\OO_q(\textbf{S}_{n})$ &  $\frac{n(n-1)}{2}$ & $\EE_{-2(n-1)}$\\
 %%%%%%%%%%%%%%%%%%%%%%%%%%%%%%%%%%%%
  $\OO_q(\mathbb{OP}^2)$ & 16 & $\EE_{-12}$\\
%%%%%%%%%%%%%%%%%%%%%%%%%%%%%%%%%%%%
  $\OO_q(\textbf{F})$  & 27 & $\EE_{-18}$\\
\hline
\end{tabular}
}
\end{table}
\end{center}
\begin{remark}
By a theorem of Atiyah, a $2m$-dimensional compact Hermitian manifold is spin if and only if its canonical line module $\Om^{(m,0)}$ admits a holomorphic square root \mbox{\cite[Proposition 3.2]{AtiyahSpinSurfaces}}. Thus from Table~\ref{table:CQFMsEk}  we see that the classical Grassmannians $\mathrm{Gr}_{s,n+1}$, and the classical Lagrangian Grassmannians $\mathbf{L}_n$, are spin for all $n \in 2\mathbb{Z}_{> 0} + 1$. Moreover, the even quadrics $\mathbf{Q}_{2n}$, and the spinor varieties $\mathbf{S}_n$, are spin, for all $n \in \mathbb{Z}_{> 0}$. For the exceptional cases, both the Cayley plane and the Freudenthal variety are spin. Atiyah's theorem suggests a definition for noncommutative Hermitian spin structures with a substantial ready-made family of noncommutative examples. 
\end{remark}

%  \bibliographystyle{abbrv}
%  \bibliography{bigbib}

\begin{thebibliography}{10}

\bibitem{ArtinZhang}
M.~Artin and J.~J. Zhang.
\newblock Noncommutative projective schemes.
\newblock {\em Adv. Math.}, 109(2):228--287, 1994.

\bibitem{AtiyahSpinSurfaces}
M.~F. Atiyah.
\newblock Riemann surfaces and spin structures.
\newblock {\em Ann. Sci. \'{E}cole Norm. Sup. (4)}, 4:47--62, 1971.

\bibitem{AHDM}
M.~F. Atiyah, V.~G. Drinfeld, N.~J. Hitchin, and Y.~I. Manin.
\newblock Construction of instantons.
\newblock {\em Phys. Lett. A}, 65(3):185--187, 1978.

\bibitem{BastonEastwood}
R.~J. Baston and M.~G. Eastwood.
\newblock {\em The {P}enrose transform. {I}ts interaction with representation
  theory}.
\newblock Oxford Mathematical Monographs. The Clarendon Press, Oxford
  University Press, New York, 1989.
\newblock Oxford Science Publications.

\bibitem{BeggsMajidChern}
E.~Beggs and S.~Majid.
\newblock Spectral triples from bimodule connections and {C}hern connections.
\newblock {\em J. Noncommut. Geom.}, 11(2):669--701, 2017.

\bibitem{BeggsMajid:Leabh}
E.~Beggs and S.~Majid.
\newblock {\em Quantum Riemannian Geometry}, volume 355 of {\em Grundlehren der
  mathematischen Wissenschaften}.
\newblock Springer International Publishing, 1 edition, 2019.

\bibitem{BS}
E.~Beggs and P.~S. Smith.
\newblock {Noncommutative complex differential geometry}.
\newblock {\em J. Geom. Phys.}, 72:7--33, 2013.

\bibitem{IrrBW}
A.~Carotenuto, F.~D\'iaz~Garc\'ia, and R.~\'O~Buachalla.
\newblock A {B}orel--{W}eil theorem for the irreducible quantum flag manifolds.
\newblock \emph{Int. Math. Res. Not. IMRN} (to appear) {arXiv preprint
  math.QA/2112.03305}.

\bibitem{KMOS}
A.~Carotenuto, C.~Mrozinski, and R.~\'O~Buachalla.
\newblock A {B}orel--{W}eil theorem for the quantum {G}rassmannians.
\newblock {arXiv preprint math.QA/1611.07969v4}.

\bibitem{BimoduleConn}
A.~Carotenuto and R.~\'O~Buachalla.
\newblock Bimodule connections for relative line modules over the irreducible
  quantum flag manifolds.
\newblock {arXiv preprint math.QA/2202.09842}.

\bibitem{SISSACPn}
F.~D'Andrea and L.~D\c{a}browski.
\newblock Dirac operators on quantum projective spaces.
\newblock {\em Comm. Math. Phys.}, 295(3):731--790, 2010.

\bibitem{DAndreaLandiMonopoles}
F.~D'Andrea and G.~Landi.
\newblock Anti-selfdual connections on the quantum projective plane: monopoles.
\newblock {\em Comm. Math. Phys.}, 297(3):841--893, 2010.

\bibitem{Antiselfdual}
F.~D'Andrea and G.~Landi.
\newblock Anti-selfdual connections on the quantum projective plane:
  instantons.
\newblock {\em Comm. Math. Phys.}, 333(1):505--540, 2015.

\bibitem{CQHKS}
B.~Das, R.~\'O~Buachalla, and P.~Somberg.
\newblock Compact quantum homogeneous {K}{\"a}hler spaces.
\newblock {arXiv preprint math.QA/1910.14007}.

\bibitem{SpectralGap}
B.~Das, R.~\'O~Buachalla, and P.~Somberg.
\newblock Spectral gaps for twisted {D}olbeault--{D}irac operators over the
  irreducible quantum flag manifolds.
\newblock (in preparation).

\bibitem{DOKSS}
F.~D\'{\i}az~Garc\'{\i}a, A.~Krutov, R.~\'{O}~Buachalla, P.~Somberg, and K.~R.
  Strung.
\newblock Holomorphic relative {H}opf modules over the irreducible quantum flag
  manifolds.
\newblock {\em Lett. Math. Phys.}, 111(10):24, 2021.

\bibitem{KoornDijk}
M.~S. Dijkhuizen and T.~H. Koornwinder.
\newblock C{QG} algebras: a direct algebraic approach to compact quantum
  groups.
\newblock {\em Lett. Math. Phys.}, 32(4):315--330, 1994.

\bibitem{DijkStok}
M.~S. Dijkhuizen and J.~Stokman.
\newblock Quantized flag manifolds and irreducible {$*$}-representations.
\newblock {\em Comm. Math. Phys. (2)}, 203:297--324, 1999.

\bibitem{DrinfeldICM}
V.~G. Drinfeld.
\newblock Quantum groups.
\newblock In {\em Proceedings of the {I}nternational {C}ongress of
  {M}athematicians, {V}ol. 1, 2 ({B}erkeley, {C}alif., 1986)}, pages 798--820.
  Amer. Math. Soc., Providence, RI, 1987.

\bibitem{DVMadoreMouradBimodule}
M.~Dubois-Violette, J.~Madore, T.~Masson, and J.~Mourad.
\newblock Linear connections on the quantum plane.
\newblock {\em Lett. Math. Phys.}, 35(4):351--358, 1995.

\bibitem{MadoreBimodule}
M.~Dubois-Violette, J.~Madore, T.~Masson, and J.~Mourad.
\newblock On curvature in noncommutative geometry.
\newblock {\em J. Math. Phys.}, 37(8):4089--4102, 1996.

\bibitem{DVMichor}
M.~Dubois-Violette and P.~W. Michor.
\newblock Connections on central bimodules in noncommutative differential
  geometry.
\newblock {\em J. Geom. Phys.}, 20(2-3):218--232, 1996.

\bibitem{VTBOOK}
H.~Esnault and E.~Viehweg.
\newblock {\em Lectures on vanishing theorems}, volume~20 of {\em DMV Seminar}.
\newblock Birkh\"{a}user Verlag, Basel, 1992.

\bibitem{FRT}
L.~D. Faddeev, N.~Y. Reshetikhin, and L.~A. Takhtadzhyan.
\newblock Quantization of {L}ie groups and {L}ie algebras.
\newblock {\em Algebra i Analiz}, 1(1):178--206, 1989.

\bibitem{HK}
I.~Heckenberger and S.~Kolb.
\newblock The locally finite part of the dual coalgebra of quantized
  irreducible flag manifolds.
\newblock {\em Proc. London Math. Soc. (3)}, 89(2):457--484, 2004.

\bibitem{HKdR}
I.~Heckenberger and S.~Kolb.
\newblock De {R}ham complex for quantized irreducible flag manifolds.
\newblock {\em J. Algebra}, 305(2):704--741, 2006.

\bibitem{HK:BGGdR}
I.~Heckenberger and S.~Kolb.
\newblock Differential forms via the {B}ernstein-{G}elfand-{G}elfand resolution
  for quantized irreducible flag manifolds.
\newblock {\em J. Geom. Phys.}, 57(11):2316--2344, 2007.

\bibitem{Humph}
J.~E. Humphreys.
\newblock {\em Introduction to {L}ie algebras and representation theory},
  volume~9 of {\em Graduate Texts in Mathematics}.
\newblock Springer-Verlag, New York-Berlin, 1978.
\newblock Second printing, revised.

\bibitem{HUY}
D.~Huybrechts.
\newblock {\em Complex geometry: an introduction}.
\newblock {U}niversitext. Springer--Verlag Berlin Heidelberg, 1 edition, 2005.

\bibitem{HEYMs}
M.~Itoh and H.~Nakajima.
\newblock Yang--{M}ills connections and {E}instein--{H}ermitian metrics.
\newblock In {\em K\"{a}hler metric and moduli spaces}, volume~18 of {\em Adv.
  Stud. Pure Math.}, pages 395--457. Academic Press, Boston, MA, 1990.

\bibitem{JantzenRepBook}
J.~C. Jantzen.
\newblock {\em Representations of algebraic groups}, volume 107 of {\em
  Mathematical Surveys and Monographs}.
\newblock American Mathematical Society, Providence, RI, second edition, 2003.

\bibitem{Jimbo1986}
M.~Jimbo.
\newblock A {$q$}-analogue of {$U(\mathfrak{gl}(N+1))$}, {H}ecke algebra, and
  the {Y}ang-{B}axter equation.
\newblock {\em Lett. Math. Phys.}, 11(3):247--252, 1986.

\bibitem{KLvSPodles}
M.~Khalkhali, G.~Landi, and W.~D. van Suijlekom.
\newblock Holomorphic structures on the quantum projective line.
\newblock {\em Int. Math. Res. Not. IMRN}, (4):851--884, 2011.

\bibitem{KSLeabh}
A.~Klimyk and K.~Schm\"udgen.
\newblock {\em Quantum Groups and Their Representations}.
\newblock Texts and Monographs in Physics. Springer-Verlag, 1997.

\bibitem{KoszulMalgrange}
J.-L. Koszul and B.~Malgrange.
\newblock Sur certaines structures fibr\'{e}es complexes.
\newblock {\em Arch. Math. (Basel)}, 9:102--109, 1958.

\bibitem{MTSUK}
U.~Kr\"{a}hmer and M.~Tucker-Simmons.
\newblock On the {D}olbeault-{D}irac operator of quantized symmetric spaces.
\newblock {\em Trans. London Math. Soc.}, 2(1):33--56, 2015.

\bibitem{KobyHit}
M.~L\"{u}bke and A.~Teleman.
\newblock {\em The {K}obayashi--{H}itchin correspondence}.
\newblock World Scientific Publishing Co., Inc., River Edge, NJ, 1995.

\bibitem{Maj}
S.~Majid.
\newblock Noncommutative {R}iemannian and spin geometry of the standard
  {$q$}-sphere.
\newblock {\em Comm. Math. Phys.}, 256:255--285, 2005.

\bibitem{WignerMas}
A.~Masuoka and D.~Wigner.
\newblock Faithful flatness of {H}opf algebras.
\newblock {\em J. Algebra}, 170(1):156--164, 1994.

\bibitem{MarcoConj}
M.~Matassa.
\newblock K\"{a}hler structures on quantum irreducible flag manifolds.
\newblock {\em J. Geom. Phys.}, 145:103477, 16, 2019.

\bibitem{MMF2}
R.~\'{O}~Buachalla.
\newblock Noncommutative complex structures on quantum homogeneous spaces.
\newblock {\em J. Geom. Phys.}, 99:154--173, 2016.

\bibitem{MMF3}
R.~\'{O}~Buachalla.
\newblock Noncommutative {K}\"{a}hler structures on quantum homogeneous spaces.
\newblock {\em Adv. Math.}, 322:892--939, 2017.

\bibitem{OSV}
R.~\'O~Buachalla, J.~\v{S}\v{t}ovi\v{c}ek, and A.-C. van Roosmalen.
\newblock A {K}odaira vanishing theorem for noncommutative {K}\"ahler
  structures.
\newblock {arXiv preprint math.QA/1801.08125.}

\bibitem{VinbergOnishchik}
A.~L. Onishchik and E.~B. Vinberg.
\newblock {\em Lie groups and algebraic groups}.
\newblock Springer Series in Soviet Mathematics. Springer-Verlag, Berlin, 1990.
\newblock Translated from the Russian and with a preface by D. A. Leites.

\bibitem{PolishSch}
A.~Polishchuk and A.~Schwarz.
\newblock Categories of holomorphic vector bundles on noncommutative two-tori.
\newblock {\em Comm. Math. Phys.}, 236(1):135--159, 2003.

\bibitem{MR541027}
M.~Raynaud.
\newblock Contre-exemple au ``vanishing theorem'' en caract\'{e}ristique
  {$p>0$}.
\newblock In {\em C. {P}. {R}amanujam---a tribute}, volume~8 of {\em Tata Inst.
  Fund. Res. Studies in Math.}, pages 273--278. Springer, Berlin-New York,
  1978.

\bibitem{SerreGAGA}
J.-P. Serre.
\newblock G\'{e}om\'{e}trie alg\'{e}brique et g\'{e}om\'{e}trie analytique.
\newblock {\em Ann. Inst. Fourier, Grenoble}, 6:1--42, 1955--1956.

\bibitem{Tak}
M.~Takeuchi.
\newblock Relative {H}opf modules---equivalences and freeness criteria.
\newblock {\em J. Algebra}, 60(2):452--471, 1979.

\bibitem{WoroCQPGs}
S.~L. Woronowicz.
\newblock Compact matrix pseudogroups.
\newblock {\em Comm. Math. Phys.}, 111(4):613--665, 1987.

\end{thebibliography}

\end{document}